\documentclass[12pt]{article}

\RequirePackage[colorlinks,citecolor=blue,urlcolor=blue]{hyperref}
\usepackage{geometry,mathrsfs}
\usepackage{amsbsy,amsmath,amsthm,amssymb,graphicx,subcaption}
\usepackage{algorithm}
\usepackage{algpseudocode}
\usepackage{verbatim}
\usepackage{natbib}
\usepackage{xcolor}

\newcommand{\df}{\mathrm{d}}
\newcommand{\X}{\mathsf{X}}
\newcommand{\Y}{\mathsf{Y}}
\newcommand{\Z}{\mathsf{Z}}
\newcommand{\PD}{P_{\tiny\mbox{DG}}}
\newcommand{\PR}{P_{\tiny\mbox{RG}}}

\newcommand{\PX}{P_{\tiny\mbox{XDG}}}
\newcommand{\PY}{P_{\tiny\mbox{YDG}}}
\newcommand{\PDM}{P_{\tiny\mbox{DC}}}
\newcommand{\PDMM}{\tilde{P}_{\tiny\mbox{DC}}}
\newcommand{\PRM}{P_{\tiny\mbox{RC}}}
\newcommand{\PRMM}{\tilde{P}_{\tiny\mbox{RC}}}
\newcommand{\PXM}{P_{\tiny\mbox{XDC}}}
\newcommand{\rhoR}{\rho(\PR)}
\newcommand{\rhoD}{\rho(\PD)}
\newcommand{\rhoRM}{\rho(\PRM)}
\newcommand{\rhoDM}{\rho(\PDM)}

\newcommand{\hgn}{\hat{g}_n}

\newcommand{\Q}{P_{\tiny\mbox{MH}}}

\newcommand{\pcite}[1]{\citeauthor{#1}'s \citeyearpar{#1}}

\theoremstyle{remark}
\newtheorem{theorem}{Theorem}[section]
\newtheorem{lemma}[theorem]{Lemma}
\newtheorem{corollary}[theorem]{Corollary}
\newtheorem{proposition}[theorem]{Proposition}

\newtheorem{assumption}[theorem]{Assumption}
\newtheorem{example}[theorem]{Example}
\newtheorem{remark}[theorem]{Remark}
\newtheorem{condition}[theorem]{Condition}

\title{Convergence Rates of Two-Component MCMC Samplers}
\date{\today}

\author{Qian Qin and Galin Jones \\
	School of Statistics \\
	University of Minnesota}

\begin{document}

	
	\maketitle
	
	\begin{abstract}
		Component-wise MCMC algorithms, including Gibbs and conditional
		Metropolis-Hastings samplers, are commonly used for sampling from
		multivariate probability distributions.  A long-standing question
		regarding Gibbs algorithms is whether a deterministic-scan
		(systematic-scan) sampler converges faster than its random-scan
		counterpart.  We answer this question when the samplers involve two
		components by establishing an exact quantitative relationship between
		the $L^2$ convergence rates of the two samplers.  The relationship
		shows that the deterministic-scan sampler converges faster.  We also
		establish qualitative relations among the convergence rates of
		two-component Gibbs samplers and some conditional Metropolis-Hastings
		variants.  For instance, it is shown that if some two-component conditional
		Metropolis-Hastings samplers are geometrically ergodic, then so are
		the associated Gibbs samplers.
	\end{abstract}

\section{Introduction}
\label{sec:intro}

Markov chain Monte Carlo (MCMC) algorithms are useful for sampling
from complicated distributions \citep{brooks2011handbook}.
Component-wise MCMC algorithms, such as Gibbs samplers and conditional
Metropolis-Hastings (CMH) samplers, sometimes called
Metropolis-within-Gibbs, are among the most useful in multivariate
settings.  We study the convergence rates of two-component Gibbs
samplers and the case where the components may be updated using
Metropolis-Hastings, paying particular attention to the relationship
between the convergence rates of the Markov chains.

Investigating the convergence rates of the underlying Markov chains is
important for ensuring a reliable simulation effort \citep{geye:1992,
  flegal2008markov, jones2001honest, vats:etal:mce:2020}.  If the
Markov chain converges sufficiently fast, then, under moment
conditions, a central limit theorem holds \citep{chan1994discussion,
  doss:etal:2014, hobert2002applicability, jone:2004,
  robe:etal:viz:2020}.  Additionally, asymptotically valid Monte Carlo
standard errors are available \citep{ dai:jone:2017, fleg:jone:2010,
  jone:etal:2006, vats:etal:sve:2018, vats:etal:moa:2019}.

Let $\Pi(\df x, \df y)$ be a joint probability distribution having
support $\X \times \Y$ and let $\Pi_{X|Y}(\df x|y)$, $y \in \Y$, and
$\Pi_{Y|X}(\df y|x)$, $x \in \X$, be full conditional distributions.
There are many potential component-wise MCMC algorithms having $\Pi$
as their invariant distribution.  When it is possible to simulate from
the conditionals, it is natural to use a Gibbs sampler. One version is
the deterministic-scan Gibbs (DG) sampler, which is now described.

\begin{algorithm}[H]
	\caption{Deterministic-scan Gibbs sampler} \label{alg:DG}
	\begin{algorithmic}[1]
		\State {\it Input:} Current value  $(X_n,Y_n) = (x,y)$.
		\State Draw $Y_{n+1}$ from $\Pi_{Y|X}(\cdot|x)$, and call the observed value $y'$.
		\State Draw $X_{n+1}$ from $\Pi_{X|Y}(\cdot|y')$. 
		\State Set $n=n+1$.
	\end{algorithmic}
\end{algorithm}

An alternative is the random-scan Gibbs (RG) sampler which is described below.

\begin{algorithm}[H]
	\caption{Random-scan Gibbs sampler with selection probability $r \in (0,1)$ } \label{alg:RG}
	\begin{algorithmic}[1]
		\State {\it Input:} Current value $(X_n,Y_n) = (x,y)$.
		
		\State Draw $U \sim \mbox{Bernoulli}(r)$, and call the observed
		value~$u$.
		
		\State If $u = 1$, draw $X_{n+1}$ from $\Pi_{X|Y}(\cdot|y)$, and
		set $Y_{n+1} = y$.
		
		\State  If $u = 0$, draw $Y_{n+1}$ from $\Pi_{Y|X}(\cdot|x)$, and set $X_{n+1} = x$.
		
		\State Set $n=n+1$.
	\end{algorithmic}
\end{algorithm}

Two-component Gibbs samplers are surprisingly useful and widely
applicable in the analysis of sophisticated Bayesian statistical
models.
In particular, they arise naturally in data augmentation
settings \citep{hobe:2011, tanner1987calculation, van2001art}.

There is abundant study of the convergence properties of Gibbs
samplers, both in the general case \citep[see][]{liu:etal:1994,
	roberts1994geometric, liu1995covariance} and for two-component Gibbs
samplers in specific statistical settings; see, among many others,
\citet{diaconis2008gibbs}, \citet{doss:hobe:2010},
\citet{ekva:jone:2019}, \citet{hobe:geye:1998},
\citet{john:jone:2008}, \citet{johnson2015geometric},
\citet{jone:hobe:2004}, \citet{khare:hobe:2013},
\citet{marchev2004geometric}, \citet{roy:2012}, \citet{tan2009block},
\citet{wang:roy:2018:polyagamma}, and \citet{wang:roy:2018:probit}.
However, there is not yet an answer to the following basic question:
which converges faster, a deterministic- or random-scan Gibbs sampler?

There exist some qualitative results related to this question
\citep[see][]{johnson2013component, tan:etal:2013}.  For instance,
\pcite{roberts1997geometric} Proposition~3.2 states that a random-scan
Gibbs sampler is uniformly ergodic whenever an associated
deterministic-scan Gibbs sampler is too. There is also literature
devoted to finding the convergence rates of various Gibbs samplers
when~$\Pi$ is Gaussian, or approximately Gaussian \citep[see,
e.g.,][]{amit1991rates, amit1996convergence, amit1991comparing,
  roberts1997updating} or in the finite discrete state space setting
\citep{fish:1996}.  However, in general, the relationship between the
convergence rates of deterministic- and random-scan Gibbs samplers is
poorly understood.

A related question is addressed by \cite{andrieu2016random}, who shows
that the DG sampler yields sample means with smaller asymptotic
variances than its random-scan counterpart, assuming that, in the RG
sampler, the selection probability is $r=1/2$ \citep[see,
also,][]{greenwood1998information}.  On the other hand, the author
remarks that making such a comparison in terms of convergence times is
unlikely to bear fruit \citep[][page 720]{andrieu2016random}.  This is
because there are examples suggesting that, when the Gibbs sampler has
a large number of components, there is no definite answer to the
question above \citep{roberts2016suprising}.

We give an exact solution to the question in the two-component
setting. Indeed, we develop a quantitative relationship between the
convergence rates of the two types of Gibbs samplers, and show that
the deterministic-scan sampler converges faster than its random-scan
counterpart no matter the selection probability in the random scan.
This result is described now, but the full details are dealt with
carefully later. The $L^2$ convergence rate of a Markov chain is a
number in $[0,1]$, with smaller rates indicating faster convergence.
Let $\rhoD$ be the $L^2$ convergence rate of the DG sampler, and,
$\rhoR$, that of the RG sampler. We show that
\begin{equation} 
\label{eq:quantitative}
\rhoR = \frac{1 + \sqrt{1 - 4r(1-r)[1-\rhoD]}}{2} .
\end{equation}
There are some easy, but noteworthy, consequences of this result.
Notice that (i) $\rhoR \in [1/2,1]$ while $\rhoD \in [0,1]$; (ii) as
either $\rhoD$ or $\rhoR$ increases so does the other; (iii) if
$\rhoD < 1$, then $\rhoR > \rhoD$, but $\rhoD=1$ if and only if
$\rhoR=1$; and (iv) the optimal selection probability for $\PR$ is
$r=1/2$ in which case
\begin{equation*} 
\rhoR = \frac{1 + \sqrt{\rhoD}}{2} .
\end{equation*}
In Section~\ref{sec:quantitative} we generalize this discussion and
show that the DG sampler converges faster even after taking into
account computation time.  Indeed, if $k_{\tiny\mbox{D}}$ and
$k_{\tiny\mbox{R}}$ are the number of iterations that can be run by DG
and RG samplers, respectively, in unit time,
then~\eqref{eq:quantitative} implies that
$\rho(\PD)^{k_{\tiny\mbox{D}}} \leq \rho(\PR)^{k_{\tiny\mbox{R}}}$ for
any selection probability.

Perhaps the most common type of MCMC sampler in applications are
conditional Metropolis-Hastings (CMH) samplers. These Markov chains
arise when it is infeasible to sample from at least one of the
conditional distributions associated with~$\Pi$ so that at least one
Metropolis-Hastings update must be used.  Assume that $\Pi_{Y|X}$ and
$\Pi_{X|Y}$, respectively, admit density functions $\pi_{Y|X}$ and
$\pi_{X|Y}$.  Let $q(\cdot|x,y), \, (x,y) \in \X \times \Y$, be a
proposal density function on~$\X$.  A deterministic-scan CMH (DC)
sampler we study is now described, but a more general algorithm is
considered in Section~\ref{sec:cmh}.

\begin{algorithm}[H]
	\caption{Deterministic-scan CMH sampler} \label{alg:DC}
	\begin{algorithmic}[1]
		\State {\it Input:} Current value  $(X_n,Y_n) = (x,y)$
		\State Draw $Y_{n+1}$ from $\Pi_{Y|X}(\cdot|x)$, and call the observed value $y'$.
		\State Draw a random element~$Z$ from $q(\cdot|x,y')$, and call the observed value~$z$.
		With probability
		\[
		a(z;x,y') = \min \left\{ 1, \frac{\pi_{X|Y}(z|y') q(x|z,y')}{\pi_{X|Y}(x|y') q(z|x,y')} \right\} ,
		\]
		set $X_{n+1} = z$;
		with probability $1 - a(z;x,y')$, set $X_{n+1} = x$.
		\State Set $n=n+1$.
	\end{algorithmic}
\end{algorithm}
There is an obvious alternative random-scan CMH (RC) sampler.
\begin{algorithm}[H]
	\caption{Random-scan CMH sampler with selection probability $r \in (0,1)$ } \label{alg:RC}
	\begin{algorithmic}[1]
		\State {\it Input:} Current value  $(X_n,Y_n) = (x,y)$.
		
		\State Draw $U \sim \mbox{Bernoulli}(r)$, and call the observed
		value~$u$.
		
		\State If $u = 1$, draw a random element~$Z$ from $q(\cdot|x,y)$,
		and call the observed value~$z$.  With probability
		\[
		a(z;x,y) = \min \left\{ 1, \frac{\pi_{X|Y}(z|y) q(x|z,y)}{\pi_{X|Y}(x|y) q(z|x,y)} \right\} ,
		\]
		set $X_{n+1} = z$; with probability $1 - a(z;x,y)$,
		set $X_{n+1} = x$. Set $Y_{n+1} = y$.
		
		\State  If $u = 0$, draw $Y_{n+1}$ from $\Pi_{Y|X}(\cdot|x)$, and set $X_{n+1} = x$.
		
		\State Set $n=n+1$.
	\end{algorithmic}
\end{algorithm}

Despite their utility, compared to Gibbs samplers there has been
little investigation of CMH Markov chains \citep{fort2003geometric,
	herb:mcke:2009, johnson2013component, jones2014convergence,
	rose:rose:2015, roberts1997geometric, robe:rose:1998} but what there
is tends not to focus on specific statistical models.  For example,
\pcite{johnson2013component} Theorem~3 states that if a deterministic scan
component-wise Markov chain is uniformly ergodic, then so is its
random-scan counterpart, thus generalizing the result proved for Gibbs
samplers by \citet{roberts1997geometric}, which was described
previously.

Both versions of Gibbs samplers are special cases of the respective
versions of CMH samplers.  Thus it is plausible that there should be
some relationship among the convergence rates of the Markov chains of
Algorithms 1--4, especially if the CMH samplers are ``close'' to the
Gibbs samplers.  There are a few results in this direction.  For
example, there are sufficient conditions which ensure that if the RG
Markov chain is geometrically ergodic, then so is the RC Markov chain
\citep[][Theorem~6]{jones2014convergence}.  However, these
relationships are not well understood in general and the following
question has not been addressed satisfactorily: if one of the four
basic component-wise samplers is geometrically ergodic, then, in
general, which of the remaining three are also geometrically ergodic?

\begin{figure}[h] 
	\centering
	\includegraphics[width=0.35\textwidth]{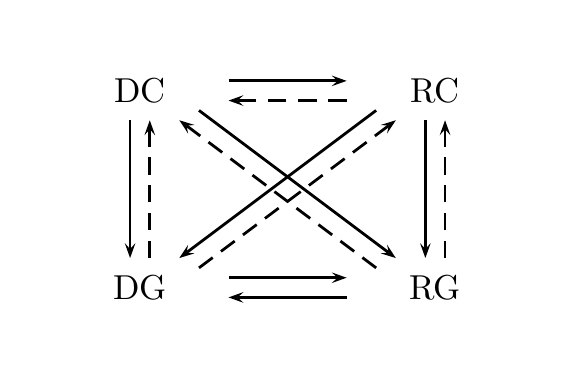}
	\caption{Relationship among
		two-component Gibbs samplers and their CMH variants in terms of
		$L^2$ geometric ergodicity.}
	\label{fig:ergodicity}
\end{figure}

We give an answer to this question by developing qualitative
relationships among the convergence rates of the DG, RG, DC, and RC
samplers, which are depicted in Figure~\ref{fig:ergodicity}.  Here, we
consider $L^2$ geometric ergodicity.  A Markov chain is $L^2$
geometrically ergodic if its $L^2$ convergence rate is strictly less
than~$1$.  Under regularity conditions, $L^2$ geometric ergodicity is
equivalent to the usual notion of geometric ergodicity defined in
terms of the total variation distance.  (This equivalence will be made
precise in Section~\ref{sec:basic}.)  In Figure~\ref{fig:ergodicity},
a solid arrow from one sampler to another means that, if the former is
$L^2$ geometrically ergodic, then so is the latter.  A dashed arrow
means that $L^2$ geometric ergodicity of the former only implies that
of the latter under appropriate conditions on the proposal density
$q(\cdot|x,y)$.  One of these conditions is condition~\ref{cond:C} in
Section~\ref{sec:qualitative}.

Figure~\ref{fig:ergodicity} yields the following.  The DG sampler is
$L^2$ geometrically ergodic if and only if the RG sampler is.  If the
RC sampler is $L^2$ geometrically ergodic for some proposal density,
then so are the DG and RG samplers.  If the DC sampler is $L^2$
geometrically ergodic for some proposal density, then so is the RC
sampler with the same proposal density. The relations depicted in
Figure~\ref{fig:ergodicity} hold regardless of the selection
probabilities for the random-scan samplers. 

The remainder is organized as follows.
Section~\ref{sec:preliminaries} contains some general theoretical
background.  In Section~\ref{sec:basic}, we lay out some basic
properties of the four types of samplers.  In
Section~\ref{sec:quantitative}, we derive~\eqref{eq:quantitative}, and
discuss its implications.  In Section~\ref{sec:qualitative}, we
establish the relations shown in Figure~\ref{fig:ergodicity} along
with additional connections with more general CMH samplers.
We give some final remarks in Section~\ref{sec:final}.  
Some technical details are relegated to the appendices.

\section{Preliminary Markov Chain Theory} 
\label{sec:preliminaries}

Let $(\Z, \mathcal{F})$ be a measurable space and
let $P$ be a Markov transition kernel (Mtk), that is, let
$P : \Z \times \mathcal{F} \to [0,1]$ be such that for each
$z \in \mathcal{Z}$, $P(z, \cdot)$ is a probability measure and for
each $A \in \mathcal{F}$, $P(\cdot, A)$ is measurable.  
For a positive integer~$n$, denote the $n$-step transition kernel associated with~$P$ by $P^n$, so that $P^1 = P$, and
\[
P^{n+1}(z,A) = \int_{\Z} P(z',A) P^n(z,\df z')
\]
for $z \in \Z$ and $A \in \mathcal{F}$.
If $\omega$ is a probability measure on $(\Z, \mathcal{F})$ and $A \in
\mathcal{F}$, define
\[
(\omega P)(A) = \int \omega(dz) P(z, A) .
\]
Say $\omega$ is invariant for $P$ if $\omega P = \omega$.  If
\begin{equation}
\label{eq:dbc}
P(z, dz') \omega(dz) = P(z', dz) \omega(dz') ,
\end{equation}
then $P$ is said to be reversible with respect to $\omega$.
Integrating both sides of the equality in \eqref{eq:dbc} shows that
$\omega$ is invariant for~$P$.

For a measurable function $f: \Z \to \mathbb{R}$ and a probability measure $\mu: \mathcal{F} \to [0,1]$, define
\[
(Pf)(z) = \int f(z') P(z, dz') \quad
\text{and} \quad
\mu f = \int_{\Z} f(z) \mu(\df z) .
\]
Assume that~$\omega$ is invariant for~$P$.
Let $L^2(\omega)$ be the set of measurable real functions~$f$ that are
square integrable with respect to $\omega$ and let $L_0^2(\omega)$ be
the set of functions~$f \in L^2(\omega)$ such that $\omega f = 0$.
For $f, g \in L^2(\omega)$, define their inner product to be
\[
\langle f, g \rangle_{\omega} = \int_{\Z} f(z) g(z) \, \omega(\df z) \,, 
\]
and let $\|f\|_{\omega}^2 = \langle f, f \rangle_{\omega}$. 
Then $(L^2(\omega), \langle\cdot,\cdot\rangle_{\omega})$ and $(L_0^2(\omega), \langle\cdot,\cdot\rangle_{\omega})$ form two real Hilbert spaces.
For any
$f \in L_0^2(\omega)$, we have $Pf \in L_0^2(\omega)$. Thus,~$P$ can
be regarded as a linear operator on $L_0^2(\omega)$.  Let
\[
\|P \|_{\omega} = \sup_{f \in L_0^2(\omega), \,\|f\|_{\omega} = 1} \|P f\|_{\omega} \,.
\]
By the Cauchy-Schwarz inequality, $\|P\|_{\omega} \leq 1$.  When~$P$
is reversible with respect to~$\omega$,~$P$, as an operator on
$L_0^2(\omega)$, is self-adjoint so that
$\langle Pf_1, f_2 \rangle_{\omega} = \langle f_1, P
f_2\rangle_{\omega}$ for $f_1,f_2 \in L_0^2(\omega)$, and
\[
\|P\|_{\omega} = \sup_{f \in L_0^2(\omega), \,\|f\|_{\omega} = 1}
|\langle Pf, f\rangle_{\omega}| .
\]
Moreover, if $P$ is self-adjoint, then for each positive integer $n$,
\[
\|P^n\|_{\omega} = \|P\|^{n}_{\omega} 
\]
\citep[see, e.g.,][\S 30 Corollary 8.1, \S 31 Corollary
2.1]{helmberg2014introduction}. Say~$P$ is non-negative definite if it
is self-adjoint, and $\langle Pf, f \rangle_{\omega} \geq 0$ for each
$f \in L_0^2(\omega)$.

For two probability measures $\mu$ and $\nu$ on $(Z, \mathcal{F})$,
define their $L^2$ (or $\chi^2$) distance to be
\[
\|\mu - \nu\|_{\omega} = \sup_{f \in L_0^2(\omega), \,\|f\|_{\omega}
	= 1} |\mu f - \nu f| \,.
\]
Let $L_*^2(\omega)$ be the set of probability measures~$\mu$ such that
$\df \mu / \df \omega \in L^2(\omega)$.  
When
$\mu, \nu \in L_*^2(\omega)$,
\[
  \|\mu - \nu\|_{\omega} = \sup_{f \in L_0^2(\omega),
    \,\|f\|_{\omega}=1}  \left\langle \frac{\df \mu}{\df \omega}
      - \frac{\df \nu}{\df \omega} , f \right\rangle_{\omega} 
  = \left\| \frac{\df \mu}{\df \omega} - \frac{\df \nu}{\df \omega}
  \right\|_{\omega} \,.
\]
The $L^2$ convergence rate of the Markov chain associated with~$P$,
denoted by $\rho(P)$, is defined to be the infimum of $\rho \in [0,1]$
such that, for each $\mu \in L_*^2(\omega)$, there exists $C_{\mu} <
\infty$ such that, for each positive integer~$n$,
\[
\|\mu P^n - \omega\|_{\omega} < C_{\mu} \rho^n .
\]
When $\rho(P) < 1$, we say that the Markov chain is $L^2$
geometrically ergodic, or more simply,~$P$ is $L^2$ geometrically
ergodic. The following is a direct consequence of
\pcite{roberts1997geometric} Theorem 2.1 and we will use it
extensively.
\begin{lemma}
	\label{lem:roberts} 
	If $P$ is reversible with respect to~$\omega$, then
	$\rho(P) = \|P\|_{\omega}$.
\end{lemma}

The following comparison lemma will be useful in conjunction with
Lemma~\ref{lem:roberts}. 
\begin{lemma}
	\label{lem:comparison}
	Let $P_1$ and $P_2$ be Mtks on
	$(\Z, \mathcal{F})$ having a common stationary distribution~$\omega$.
	Suppose further that $\|P_2\|_{\omega} < 1$ and there exists
	$\delta > 0$ such that, for $z \in \Z$ and $A \in \mathcal{F}$,
	$P_1(z,A) \geq \delta P_2(z,A)$.  Then $\|P_1\|_{\omega} < 1$.
\end{lemma}
\begin{proof}
	Without loss of generality, assume that $\delta < 1$.  Let
	$R(z,A) = (1-\delta)^{-1} (P_1(z,A) - \delta P_2(z,A))$.  Then
	$R(z,A)$ defines an Mtk such that
	$\omega R = \omega$.  By Cauchy-Schwarz, $\|R\|_{\omega} \leq 1$.
	By the triangle inequality,
	$\|P_1\|_{\omega} \leq \delta \|P_2\|_{\omega} + (1-\delta)
	\|R\|_{\omega} < 1$.
\end{proof}

We can use these lemmas to obtain a generalization of
\pcite{jones2014convergence} Proposition 2.  This will allow us to
treat the selection probabilities in the random-scan algorithms as
arbitrary when studying their qualitative convergence rates.
\begin{proposition} 
	\label{pro:selection}
	Let $P_1$ and $P_0$ be Mtks on
	$(\Z, \mathcal{F})$ such that for any
	$0<r<1$ the mixture kernel $P_{r} = rP_{1} + (1-r) P_{0}$ is
	reversible with respect to $\omega$. If $\rho(P_{r_0}) <1$ for some
	$r_0 \in (0,1)$, then $\rho(P_{r}) <1$ for every $r \in (0,1)$.
\end{proposition}
\begin{proof}
	For each $z \in \Z$ and $A \in \mathcal{F}$,
	\[
	P_r (z,A) \geq \min \left\{ \frac{r}{r_0},
	\frac{1-r}{1-r_0} \right\} P_{r_0} (z,A) \,.
	\]
	Since $P_{r}$ is reversible with respect to $\omega$ for all
	$r \in (0,1)$, the claim follows from Lemmas~\ref{lem:roberts}
	and~\ref{lem:comparison}.
\end{proof}

\begin{remark}
While we will not require it, it is straightforward to extend the
proof of Proposition~\ref{pro:selection} to the setting where there
is an arbitrary, but finite, number of Mtks in the mixture.
\end{remark}

We are now in position to begin our study of the algorithms defined in
Section~\ref{sec:intro}.

\section{Basic Properties of Two-component Samplers}
\label{sec:basic} 

We begin by defining the Markov transition kernels for the four
algorithms described in Section~\ref{sec:intro} along with some
related Markov chains that will be useful later.  Then we will turn
our attention to some basic properties of the operators and total
variation norms for these Markov chains.

Suppose $(\X \times \Y, \mathcal{F}_{X} \times\mathcal{F}_{Y})$ is a
measurable space with a joint probability
distribution $\Pi (\df x, \df y)$.  Let $\Pi_{X}(\df x)$ and
$\Pi_{Y}(\df y)$ be the associated marginal distributions, and,
$\Pi_{X|Y}(\df x|y)$ and $\Pi_{Y|X}(\df y | x)$, the full
conditional distributions. To avoid trivial cases we make the
following standing assumption.
\begin{assumption} 
	\label{as:1} 
	There exist $A_1, A_2 \in \mathcal{F}_{X} $ and
	$B_1, B_2 \in \mathcal{F}_{Y} $ such that $A_1 \cap A_2 = \emptyset$,
	$B_1 \cap B_2 = \emptyset$, and that $\Pi_X(A_1) > 0$,
	$\Pi_X(A_2) > 0$, $\Pi_Y(B_1) > 0$, $\Pi_Y(B_2)>0$.
\end{assumption}
When Assumption~\ref{as:1} is violated, at least one of $\mathcal{F}_{X} $ and
$\mathcal{F}_{Y}$ contain only sets of measure zero or one, and
all the problems we study become essentially trivial. 

Letting~$\Pi$,~$\Pi_X$, or~$\Pi_Y$ play the role of $\omega$ from
Section~\ref{sec:preliminaries}, as appropriate, allows us to consider
the Mtks defined in the sequel as linear operators on the appropriate
Hilbert spaces.  Assumption~\ref{as:1} ensures $L_0^2(\Pi)$,
$L_0^2(\Pi_X)$, and $L_0^2(\Pi_Y)$ contain non-zero elements.

\subsection{Markov transition kernels}
\label{sec:mtk}

The Mtk for the DG sampler is
\[
\PD((x,y), (\df x', \df y')) = \Pi_{X|Y}(\df x'|y') \Pi_{Y|X}(\df
y'|x) .
\]
Now $\PD$ has $\Pi$ as its invariant distribution, but it is not
reversible with respect to $\Pi$. If $\delta_x$ and $\delta_y$ are
point masses at~$x$ and~$y$, respectively, then the Mtk for the RG
sampler is
\[
\PR((x,y), (\df x', \df y')) = r \Pi_{X|Y}(\df x'|y) \delta_y(\df y')
+ (1-r) \Pi_{Y|X}( \df y'|x) \delta_x(\df x') .
\]
It is well known that $\PR$ is reversible with respect to $\Pi$ and
hence has $\Pi$ as its invariant distribution. Now let $\Q$ denote the
Metropolis-Hastings Mtk \citep{tierney1994markov, tier:1998} which is reversible
with respect to the full conditional $\Pi_{X|Y}$.  Then the Mtk for
the DC sampler is
\[
\PDM((x,y), (\df x', \df y')) =  \Q(\df x'|x,y') \, \Pi_{Y|X}(\df y'|x) .
\]
Note that $\PDM$ has $\Pi$ as its invariant distribution, but it is
not reversible with respect to $\Pi$.
The Mtk for the RC sampler is
\[
\PRM((x,y), (\df x', \df y')) = r \Q(\df x'|x,y) \delta_y(\df y') +
(1-r) \Pi_{Y|X}(\df y'|x) \delta_x(\df x') 
\]
and it is again well known that $\PRM$ is reversible with respect to
$\Pi$ and hence has $\Pi$ as its invariant distribution.

It will be convenient to consider marginalized versions of the DG chain, which we now define. The $X$-marginal DG chain is defined
on~$\X$, and its Mtk is
\[
\PX(x,\df x') = \int_{\Y} \Pi_{X|Y}(\df x'|y) \Pi_{Y|X}(\df y|x)  .
\]
Similarly, the $Y$-marginal DG chain is defined on~$\Y$, and has Mtk
\[
\PY(y, \df y') = \int_{\X} \Pi_{Y|X}(\df y'|x) \Pi_{X|Y}(\df x|y) .
\]
Note that $\PX$ and $\PY$ are reversible with respect to $\Pi_X$ and
$\Pi_Y$, respectively \citep[][Lemma~3.1]{liu:etal:1994}.  Moreover, it is
well-known that the convergence properties of the marginal, $\PX$ and
$\PY$, chains are essentially those of the original DG chain
\citep{robe:1995, roberts2001markov}.

There also exists an $X$-marginal version of the DC sampler (but not a
$Y$-marginal version) with Mtk given by
\[
\PXM(x, \df x') = \int_{\Y} \Q(\df x'|x,y) \Pi_{Y|X}(\df y|x) .
\]
\pcite{jones2014convergence} Section~2.4 shows that $\PXM$ is reversible with
respect to $\Pi_X$ and enjoys the same qualitative rate of convergence
in total variation norm as the parent DC sampler.

\subsection{Operator norms}
\label{sec:norms}

It is clear that $\PD$, $\PR$, $\PDM$, and $\PRM$ can be regarded as
operators defined on $L_0^2(\Pi)$.  Among them,~$\PR$ and~$\PRM$ are
self-adjoint.  It can be checked that $\PR$ is non-negative definite; see \cite{liu1995covariance}, Lemma~3, and \cite{rudo:ullr:2013}, Section 3.2.
Also, $\PX$ and~$\PXM$ are self-adjoint
operators on $L_0^2(\Pi_X)$, while~$\PY$ is a self-adjoint operator on
$L_0^2(\Pi_Y)$.  Moreover, $\PX$ and~$\PY$ are non-negative definite
\citep[][Lemma~3.2]{liu:etal:1994}.

Using Lemma~\ref{lem:roberts} and the fact that RG and RC chains are
reversible with respect to $\Pi$, we have
\[
\rho(\PR) = \|\PR\|_{\Pi} \quad  \text{and} \quad \rho(\PRM) = \|\PRM\|_{\Pi} .
\]
Similar relations for the deterministic-scan samplers are given in the
following lemma, whose proof is given in Appendix~\ref{app:marginal}.  
\begin{lemma} 
	\label{lem:marginal}
	For each positive integer~$n$,
	\[
	\|\PD^n\|_{\Pi}^{1/(n-1/2)} = \rho(\PD) = \|\PX\|_{\Pi_X} = \|\PY\|_{\Pi_Y} \,
	\]
	\[
	\|\PDM^n\|_{\Pi}^{1/(n-1)} \leq \rho(\PDM) = \|\PXM\|_{\Pi_X} \leq \|\PDM^n\|_{\Pi}^{1/n} \,.
	\]
	($\|\PDM^n\|_{\Pi}^{1/(n-1)}$ is interpreted as~$0$ when $n = 1$.)
\end{lemma}

The surprising exponent $1/(n-1/2)$ in the above lemma has also
appeared in related results on alternating projections
\citep[][Theorem 2]{kayalar1988error}.

Applying Lemmas~\ref{lem:roberts} and~\ref{lem:marginal} we obtain the
following.

\begin{corollary}
	\label{cor:marginal}
	$\rhoD = \rho(\PX) = \rho(\PY)$ and $\rhoDM = \rho(\PXM)$.
\end{corollary}

For $g \in L_0^2(\Pi_X)$ and $h \in L_0^2(\Pi_Y)$, let
\[
\gamma(g,h) = \int_{\X \times \Y} g(x) h(y) \Pi(\df x, \df y) 
\]
and
\[
\bar{\gamma} = \sup \{ \gamma(g,h) : \, g \in L_0^2(\Pi_X), \|g\|_{\Pi_X} = 1, h \in L_0^2(\Pi_Y), \|h\|_{\Pi_Y} = 1 \}.
\]
We say that $\bar{\gamma} \in [0,1]$ is the maximal correlation
between~$X$ and~$Y$.  The following result can be found in Theorem 3.2
of \cite{liu:etal:1994}; see also \cite{liu1995covariance} and
\cite{vidav1977norm}.

\begin{lemma} \label{lem:maximal}
	\[
	\bar{\gamma}^2 = \|\PX\|_{\Pi_X} = \|\PY\|_{\Pi_Y} \,.
	\]
\end{lemma}

\subsection{Total variation}
\label{sec:tv}

We consider the connection between $L^2$ geometric ergodicity and the
usual notion of geometric ergodicity defined through the total
variation norm, denoted by $\|\cdot \|_{\tiny\mbox{TV}}$. For the four
component-wise Markov chains considered here we can use results from
\citet{roberts2001geometric} to show that these concepts are
equivalent \citep[see also][]{roberts1997geometric}.  A proof is
provided in Appendix~\ref{app:l1l2}.

\begin{proposition} 
	\label{pro:l1l2} 
	Let $P$ denote the Mtk for any of the DG, RG, DC, and RC
	Markov chains.  Suppose that the $\sigma$-algebras $\mathcal{F}_X$ and $\mathcal{F}_Y$ are countably generated, and~$P$ is $\varphi$-irreducible.  Then $P$
	is $L^2$-geometrically ergodic if and only if it is $\Pi$-almost
	everywhere geometrically ergodic in the sense that, for $\Pi$-almost
	every $(x,y)$, there exist $C (x,y)$ and $t < 1$ such that, for all
	$n$,
	\[
	\|P^n((x,y), \cdot) - \Pi(\cdot)\|_{\tiny\mbox{TV}} \leq C(x,y) t^n .
	\]
\end{proposition}

Direct applications of Theorem~2 of \cite{roberts2001geometric} yield analogous results for the marginal chains defined by $\PX$, $\PY$, and $\PXM$.


\section{Quantitative Relationship between $\rhoD$ and $\rhoR$} 
\label{sec:quantitative}

\subsection{Main result}

The proof of the following is given in Section~\ref{ssec:proof}.
\begin{theorem} 
	\label{thm:main}
	\[
	\rhoR = \frac{1 + \sqrt{1 - 4r(1-r)[1-\rhoD]}}{2} \,.
	\]
\end{theorem}
We illustrate Theorem~\ref{thm:main} in two examples.	
\begin{figure}[h]
	\centering
	\includegraphics[width=0.8\textwidth]{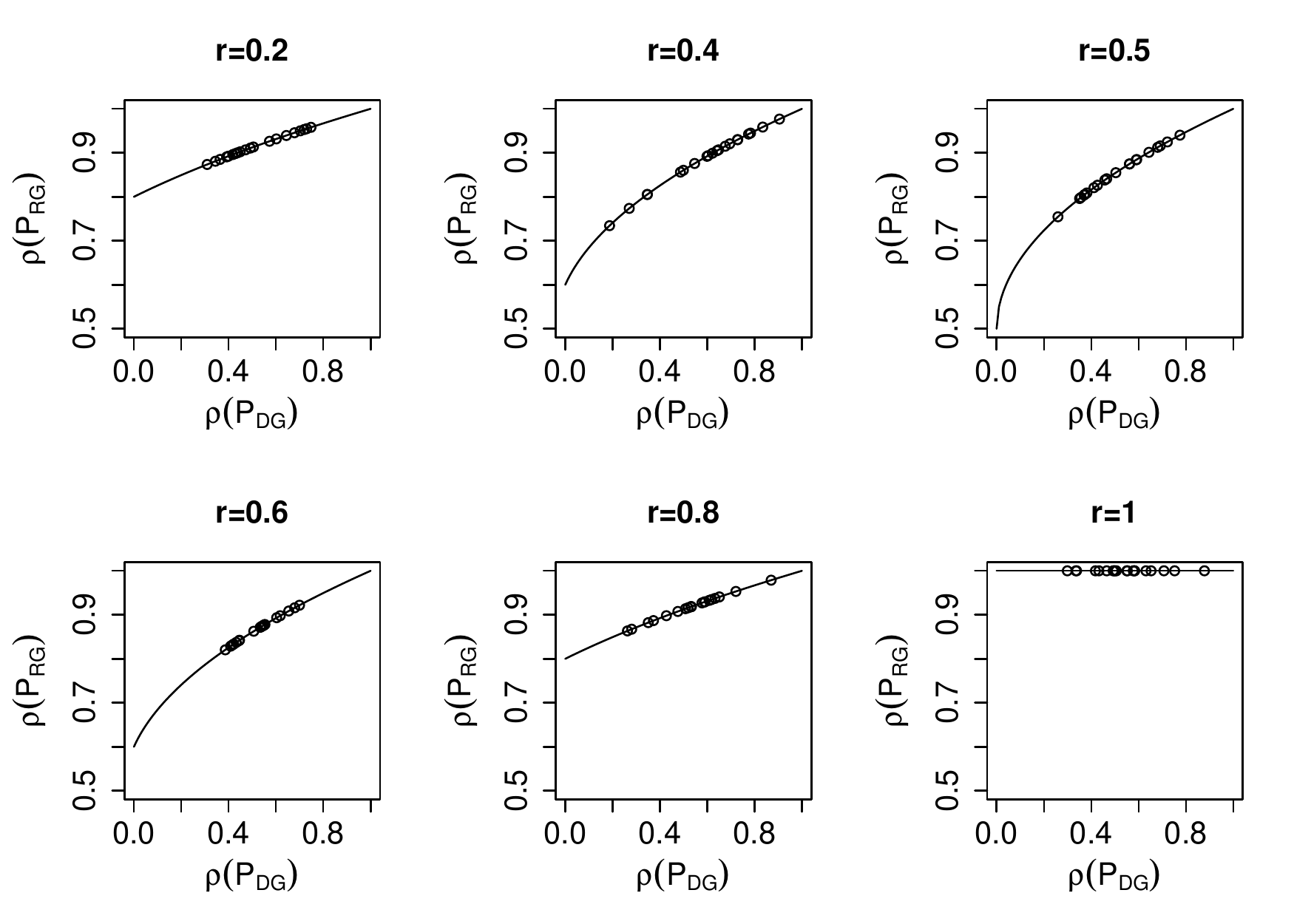}
	\caption{
		Relationship between~$\rhoD$ and~$\rhoR$ for discrete target
		distributions.  
		In each subplot, 20 joint pmfs are randomly
		generated using Dirichlet distributions with random concentration parameters.  
		Each circle corresponds
		to a joint pmf.  
		The solid curves depict the relationship given in
		Theorem~\ref{thm:main}.} \label{fig:randomscanline}
\end{figure}

\begin{example}
	When~$\Pi$ is Gaussian, there are explicit formulas for~$\rhoD$
	and~$\rhoR$ \citep{amit1996convergence,roberts1997updating}.  In
	particular, when~$\Pi$ is a bivariate Gaussian, and the correlation
	between~$X$ and~$Y$ is $\gamma \in [-1,1]$, it is well-known that
	$\rho(\PD) = \gamma^2$ \citep[see, e.g.,][Section~4.3]{diaconis2008gibbs}.
	Meanwhile,
	\[
	\rho(\PR) = \frac{1 + \sqrt{1 - 4r(1-r)(1-\gamma^2)}}{2}
	\]
	\citep[][page 193]{levine2008comment}. This is in accordance with the general
	result in Theorem~\ref{thm:main}.
\end{example}	

\begin{example}
	When $\X \times \Y$ is a finite set,~$\Pi$ can be written in the
	form of a probability mass function (pmf),~$\pi(\cdot,\cdot)$.  For
	illustration, take $\X = \Y = \{1,2,3,4,5\}$, and generate the
	elements of $\pi(i,j), \, (i,j) \in \X \times \Y,$ via a Dirichlet
	distribution.  The convergence rates of DG and RG samplers can then
	be calculated using the second largest eigenvalue in modulus of their transition matrices.  We
	repeat this experiment 20 times for different values of selection
	probabilities.  The results are displayed in
	Figure~\ref{fig:randomscanline}.
\end{example}

We now turn our attention to some of the implications of
Theorem~\ref{thm:main}.  Notice that, given the selection probability
$r \in (0,1)$,~$\rhoR$ and~$\rhoD$ are monotonic functions of each
other.  Indeed, in light of Lemmas~\ref{lem:marginal}
and~\ref{lem:maximal}, given any selection probability, the
convergence rates of the two types of Gibbs chains are completely
determined by the maximal correlation between~$X$ and~$Y$. 
$\rhoD = \bar{\gamma}^2 = 0$ if and only if $\rhoR = \max \{r, 1-r\}$;
$\rhoD = \bar{\gamma}^2 = 1$ if and only if $\rhoR = 1$; and when
$\rhoD = \bar{\gamma}^2 \in (0,1)$, $\rhoR \in (\max \{r, 1-r\},1)$
and $\rhoD < \rhoR$.

Let $k_* > 0$ be such that $\rhoR^{k_*} = \rhoD$ so that, roughly
speaking, one iteration of the DG sampler is ``worth"
$k_*$ iterations of the RG sampler in terms of convergence
rate.  By Young's inequality,
\[
\begin{aligned}
\rhoR &= \frac{1 + \sqrt{ 1-4r(1-r) + 4r(1-r)\rhoD }}{2} \\
& \geq [1-4r(1-r) + 4r(1-r)\rhoD]^{1/4} \\
& \geq \rhoD^{r(1-r)} .
\end{aligned}
\]
Therefore, $k_* \geq 1/[r(1-r)]$.

Let $t_1$ and $t_2$ be the time it takes to sample from $\Pi_{X|Y}$
and $\Pi_{Y|X}$, respectively.  For simplicity, assume that they are
constants.  Suppose that, within unit time, one can run
$k_{\tiny \mbox{D}}$ iterations of the DG sampler, and
$k_{\tiny\mbox{R}}$ iterations of the RG sampler.  Then
\[
\frac{k_{\tiny\mbox{R}}}{k_{\tiny \mbox{D}}} = \frac{t_1 + t_2}{rt_1 + (1-r)t_2} \leq \frac{1}{\min\{r,1-r\}} \leq \frac{1}{r(1-r)}  \,.
\]
Since $k_* \geq 1/[r(1-r)]$,
\[
\rhoD^{k_{\tiny \mbox{D}}} = \rhoR^{k_* k_{\tiny \mbox{D}}} =
\exp \left\{ [\log \rhoR ] k_* k_{\tiny\mbox{R}} \frac{rt_1 +
	(1-r)t_2}{t_1 + t_2} \right\} \leq \rhoR^{k_{\tiny\mbox{R}}} \,.
\]
In this sense, the DG sampler converges faster than
its random-scan counterpart.

\begin{remark}
\label{rem:more}
  It is natural to consider whether Theorem~\ref{thm:main} can be
  extended to the case where there are more than two components.  This
  proves to be challenging.  Two-component deterministic-scan Gibbs
  samplers have reversible marginal chains and hence $\PX$ and
  $\PY$, as linear operators, are self-adjoint.  Consequentially,
  there are results like Lemma~\ref{lem:marginal}, which link the
  convergence rate $\rho(\PD)$ to the norms of $\PX$ and $\PY$.
  However, deterministic-scan Gibbs samplers with more than two
  components do not possess reversible marginal chains.  Therefore,
  techniques employed herein cannot be readily applied to compare the
  convergence rates of DG and RG samplers when the number of
  components exceed two. 
\end{remark}

\subsection{Proof of Theorem~\ref{thm:main}} 
\label{ssec:proof}

\begin{remark}
	Before we begin the proof, we note that the result of
	Theorem~\ref{thm:main} is related to the theory of two projections
	\citep{bottcher2010gentle}.
	When $r=1/2$, an alternative proof of Theorem~\ref{thm:main} is
	available if we apply results on the norm of the sum of two
	projections \citep[e.g.,][Theorem 7]{duncan1976norm} along with Lemma~\ref{lem:marginal}.
\end{remark}

By Lemmas~\ref{lem:marginal} and~\ref{lem:maximal},
$\rho(\PD) = \bar{\gamma}^2$, where $\bar{\gamma} \in [0,1]$ is the
maximal correlation between~$X$ and~$Y$.  To prove
Theorem~\ref{thm:main}, we need to connect $\rho(\PR)$
to~$\bar{\gamma}$. We begin with a preliminary result.
\begin{lemma} \label{lem:lowerbound}
	\[
	\rhoR \geq \max \{ 1- r + r \bar{\gamma}^2, r + (1-r)
	\bar{\gamma}^2 \} \,.
	\]
\end{lemma}
\begin{proof}
	Let $g \in L_0^2(\Pi_X)$ be such that $\|g\|_{\Pi_X} = 1$.  Let
	$f_g$ be such that $f_g(x,y) = g(x)$ for each
	$(x,y) \in \X \times \Y$ so that $f_g \in L_0^2(\Pi)$, and
	$\|f_g\|_{\Pi}=1$. By Cauchy-Schwarz,
	\begin{equation} 
	\label{ine:PR-upper-0}
	\begin{aligned}
	\|\PR\|_{\Pi} &\geq \langle \PR f_g, f_g \rangle_{\Pi} \\
	& = r \int_{\X \times \Y} \left( \int_{\X} g(x') \Pi_{X|Y}(\df
	x'|y) \right) g(x) \Pi(\df x, \df y) + (1-r) \langle g, g \rangle_{\Pi_X} \\
	& = r \langle \PX g, g \rangle_{\Pi_X} + 1 - r \,.
	\end{aligned}
	\end{equation}
	Recall that $\PX$ is non-negative definite.  This implies that
	\[
	\|\PX\|_{\Pi_X} = \sup \{ \langle \PX g', g' \rangle_{\Pi_X}: g' \in
	L_0^2(\Pi_X), \|g'\|_{\Pi_X} = 1\} \,.
	\]
	\citep[See, e.g.,][\S 14 Corollary~5.1.]{helmberg2014introduction}
	Taking the supremum with respect to~$g$ in~\eqref{ine:PR-upper-0}
	yields
	\begin{equation} 
	\label{ine:PR-upper-1}
	\|\PR\|_{\Pi} \geq 1 - r + r \|\PX\|_{\Pi_X} = 1 - r + r \bar{\gamma}^2 \,,
	\end{equation}
	where the last equality follows from Lemma~\ref{lem:maximal}.
	
	By an analogous argument,
	\begin{equation} 
	\label{ine:PR-upper-2} 
	\|\PR\|_{\Pi} \geq r +
	(1-r)\|\PY\|_{\Pi_Y} = r + (1-r)\bar{\gamma}^2 \,.
	\end{equation}
	Recall that $\rhoR = \|\PR\|_{\Pi}$. The proof is completed by
	combining~\eqref{ine:PR-upper-1} and~\eqref{ine:PR-upper-2}.
\end{proof}

Our proof of Theorem~\ref{thm:main} hinges on the fact that, for each
$f \in L_0^2(\Pi)$ and $(x,y) \in \X \times \Y$, $\PR f(x,y)$ can be written in the
form of $g(x) + h(y)$, where
\[
g(x) = (1-r) \int_{\Y} f(x,y') \Pi_{Y|X}(\df y'|x) \, , \quad h(y) =
r \int_{\X} f(x',y) \Pi_{X|Y}(\df x'|y) \,.
\]
As we will see, this allows us to restrict our attention to a
well-behaved subspace of $L_0^2(\Pi)$ when studying the norm of $\PR$.

For $g \in L_0^2(\Pi_X)$ and $h \in L_0^2(\Pi_Y)$, let $g \oplus h$ be
the function on $\X \times \Y$ such that
\[
(g \oplus h)(x,y) = g(x) + h(y)
\]
for $(x,y) \in \X \times \Y$ (in a $\Pi$-almost everywhere sense). Let
\[
H = \{g \oplus h : \, g \in L_0^2(\Pi_X), \, h \in L_0^2(\Pi_Y) \} .
\]
Then~$H$, equipped with the inner product
$\langle \cdot, \cdot \rangle_{\Pi}$, is a subspace of $L_0^2(\Pi)$.
For $g \oplus h \in H$,
\[
\|g \oplus h\|_{\Pi}^2 = \|g\|_{\Pi_X}^2 + \|h\|_{\Pi_Y}^2 +
2\gamma(g,h) ,
\]
where $\gamma(g,h)$ is defined in Section~\ref{sec:mtk}.  It follows
that
\begin{equation} \label{ine:H-norm} (1-\bar{\gamma})(\|g\|_{\Pi_X}^2 +
\|h\|_{\Pi_Y}^2) \leq \|g \oplus h\|_{\Pi}^2 \leq
(1+\bar{\gamma})(\|g\|_{\Pi_X}^2 + \|h\|_{\Pi_Y}^2) \,.
\end{equation}
When $\bar{\gamma} < 1$, $g \oplus h = 0$ if and only if $g=0$ and
$h=0$.  It follows that, whenever $\bar{\gamma} < 1$, for any
$f \in H$, the decomposition $f = g \oplus h$ is unique.


To proceed, we present two technical results concerning~$H$.
Lemma~\ref{lem:H-closed} is proved in Appendix~\ref{app:H-closed}, and Lemma~\ref{lem:H-limit} is a direct consequence of~\eqref{ine:H-norm}.
\begin{lemma} 
	\label{lem:H-closed}
	If $\bar{\gamma} < 1$, then $H$ is a Hilbert space.
\end{lemma}

\begin{lemma} 
	\label{lem:H-limit}
	Let $\bar{\gamma} < 1$ and suppose that $\{g_n\}_{n=1}^{\infty}$ and
	$\{h_n\}_{n=1}^{\infty}$ are sequences in $L_0^2(\Pi_X)$ and
	$L_0^2(\Pi_Y)$, respectively. For $g \in L_0^2(\Pi_X)$ and
	$h \in L_0^2(\Pi_Y)$,
	$
	\lim\limits_{n \to \infty} (g_n \oplus h_n) = g \oplus h
	$
	if and only if
	$
	\lim\limits_{n \to \infty} g_n = g \,, \text{ and } \lim\limits_{n \to \infty} h_n = h \,.
	$
\end{lemma}
It is easy to check that, for every $f \in L_0^2(\Pi)$,
$\PR f \in H$.  Define $\PR|_H$ to be~$\PR$ restricted to~$H$.  The
norm of $\PR|_H$ is
\[
\|\PR|_H \|_{\Pi} = \sup_{f \in H, \, \|f\|_{\Pi} = 1} \|\PR f\|_{\Pi} \,.
\]
We then have the following lemma.
\begin{lemma} \label{lem:H} 
	\[
	\|\PR\|_{\Pi} = \|\PR|_H \|_{\Pi} \,.
	\]
\end{lemma}
\begin{proof}			
	It is clear that
	\begin{equation} \label{ine:PR-H-1}
	\|\PR\|_{\Pi} \geq \|\PR|_H\|_{\Pi} \,.
	\end{equation}
	Because the range of~$\PR$ is in~$H$, for any $f \in L_0^2(\Pi)$ and
	positive integer~$n$,
	\[
	\|\PR^n f\|_{\Pi} = \|\PR|_H^{n-1} \PR f\|_{\Pi} \leq \|\PR|_H \|_{\Pi}^{n-1} \|f\|_{\Pi} \,.
	\]
	Note that we have used the fact that $\|\PR\|_{\Pi} \leq 1$.  Since
	$\PR$ is self-adjoint, for each positive integer~$n$,
	$\|\PR^n\|_{\Pi} = \|\PR\|_{\Pi}^n$.  It follows that
	\begin{equation} \label{ine:PR-H-2} \|\PR\|_{\Pi} = \lim\limits_{n \to
		\infty} \|\PR^n\|_{\Pi}^{1/n} \leq \lim\limits_{n \to \infty}
	\|\PR|_H\|_{\Pi}^{(n-1)/n} = \|\PR|_H\|_{\Pi} \,.
	\end{equation}
	Combining~\eqref{ine:PR-H-1} and~\eqref{ine:PR-H-2} yields the desired result.
\end{proof}

We are now ready to prove the theorem.
\begin{proof}[Proof of Theorem~\ref{thm:main}]	
	
	When $\bar{\gamma} = 1$, the theorem follows from
	Lemma~\ref{lem:lowerbound} and the fact that
	$\rhoR = \|\PR\|_{\Pi} \leq 1$.  
	Assume that $\bar{\gamma} <
	1$. We first show that
	\begin{equation} 
	\label{ine:rhoR-upper} 
	\rhoR \leq \frac{1 + \sqrt{1 - 4r(1-r)(1-\bar{\gamma}^2)}}{2} \,.
	\end{equation}
	It follows from Lemma~\ref{lem:H} that $\rhoR = \|\PR|_H\|_{\Pi}$.
	Note that $\PR|_H$ is a non-negative definite operator on~$H$.  By
	Lemma~\ref{lem:approxeigen} in Appendix~\ref{app:main},~$\rhoR$ is an
	approximate eigenvalue of $\PR|_H$, that is, there exists a sequence
	of functions $\{g_n \oplus h_n\}_{n=1}^{\infty}$ in~$H$ such that
	$\|g_n \oplus h_n\|_{\Pi} = 1$ for each~$n$, and
	\begin{equation} 
	\label{eq:approxeigen}
	\lim\limits_{n \to \infty} [\PR (g_n \oplus h_n) - \rhoR (g_n \oplus h_n) ] = 0 \,.
	\end{equation}
	For every positive integer~$n$,
	\[
	\PR(g_n \oplus h_n) = [(1-r) g_n + (1-r) Q_1 h_n] \oplus (r Q_2 g_n
	+ r h_n) \,,
	\]
	where $Q_1: L_0^2(\Pi_Y) \to L_0^2(\Pi_X)$ and
	$Q_2: L_0^2(\Pi_X) \to L_0^2(\Pi_Y)$ are bounded linear
	transformations such that, for $g \in L_0^2(\Pi_X)$ and
	$h \in L_0^2(\Pi_Y)$,
	\[
	(Q_1 h)(x) = \int_{\Y} h(y) \Pi_{Y|X}(\df y|x) \,, \quad (Q_2 g)(y)
	= \int_{\X} g(x) \Pi_{X|Y}(\df x|y) \,.
	\]
	By Lemma~\ref{lem:H-limit}, \eqref{eq:approxeigen} implies that
	\begin{equation} 
	\label{eq:approxeigen-1}
	\begin{aligned}
	\lim\limits_{n \to \infty} \{[1 - r - \rhoR] g_n + (1-r) Q_1 h_n\}  = 0 \,,\\
	\lim\limits_{n \to \infty} \{[r - \rhoR]h_n + r Q_2 g_n \} = 0 \,.
	\end{aligned}
	\end{equation}
	Applying $Q_1$ to the second equality in~\eqref{eq:approxeigen-1}
	yields
	\[
	\lim\limits_{n \to \infty} \{[r - \rhoR] Q_1 h_n + r \PX g_n \} = 0
	\,.
	\]
	Subtracting (a multiple of) this from (a multiple of) the first equality in~\eqref{eq:approxeigen-1}
	gives
	\begin{equation} \label{eq:approxeigen-2}
	\lim\limits_{n \to \infty} \{ [1 - r - \rhoR] [r - \rhoR] g_n - r(1-r)
	\PX g_n \} = 0 \,.
	\end{equation}
	Similarly, applying $Q_2$ to the first equality in~\eqref{eq:approxeigen-1} and subtracting it from the second equality in~\eqref{eq:approxeigen-1} yields
	\begin{equation} \label{eq:approxeigen-3}
	\lim\limits_{n \to \infty} \{[1 - r - \rhoR] [r - \rhoR] h_n - r(1-r)
	\PY h_n \} = 0 \,.
	\end{equation}
	By Lemma~\ref{lem:maximal}, $\|\PX g_n\|_{\Pi_X} \leq \bar{\gamma}^2 \|g_n\|_{\Pi_X}$, $\|\PY h_n\|_{\Pi_Y} \leq \bar{\gamma}^2 \|h_n\|_{\Pi_Y}$.
	It follows from~\eqref{eq:approxeigen-2} and~\eqref{eq:approxeigen-3} that
	\begin{equation} \nonumber
	\begin{aligned}
	\limsup_{n \to \infty} \{[1 - r - \rhoR][r - \rhoR] - r(1-r) \bar{\gamma}^2 \} \|g_n\|_{\Pi_X} &\leq 0 \,, \\
	\limsup_{n \to \infty} \{[1 - r - \rhoR][r - \rhoR] - r(1-r) \bar{\gamma}^2 \} \|h_n\|_{\Pi_Y} &\leq 0 \,.
	\end{aligned}
	\end{equation}
	In particular,
	\begin{equation} \label{ine:approxeigen-4}
	\limsup_{n \to \infty} \{[1 - r - \rhoR][r - \rhoR] - r(1-r) \bar{\gamma}^2 \} (\|g_n\|_{\Pi_X} + \|h_n\|_{\Pi_Y}) \leq 0 \,.
	\end{equation}
	By the triangle inequality,
	\[
	\|g_n\|_{\Pi_X} + \|h_n\|_{\Pi_Y} = \|g_n \oplus 0\|_{\Pi} + \|0 \oplus h_n\|_{\Pi} \geq \|g_n \oplus h_n\|_{\Pi} = 1 \,.
	\]
	It then follows from~\eqref{ine:approxeigen-4} that
	\begin{equation} \label{ine:determ}
	[\rhoR + r - 1][\rhoR - r] - r(1-r) \bar{\gamma}^2 \leq 0 \,.
	\end{equation}
	This
	proves~\eqref{ine:rhoR-upper}.
	
	%
	%
	
	Next, we show that
	\begin{equation} 
	\label{ine:rhoR-lower-0} 
	\rhoR \geq \frac{1 + \sqrt{1 - 4r(1-r)(1-\bar{\gamma}^2)}}{2} \,.
	\end{equation}
	This will complete the proof, since, by Lemmas~\ref{lem:marginal}
	and~\ref{lem:maximal}, $\rhoD = \bar{\gamma}^2$. If
	$\bar{\gamma} = 0$, then \eqref{ine:rhoR-lower-0} follows immediately
	from Lemma~\ref{lem:lowerbound}. Now assume $\bar{\gamma} \in (0,1)$.
	Recall that $\rhoR = \|\PR\|_{\Pi}$.  It suffices to show that
	\begin{equation} 
	\label{ine:rhoR-lower}
	\|\PR\|_{\Pi} \geq \frac{1 + \sqrt{1 - 4r(1-r)(1-\bar{\gamma}^2)}}{2}.
	\end{equation}
	Recall that $\PX$ is non-negative definite, and
	$\|\PX\|_{\Pi_X} = \bar{\gamma}^2$.  Hence,~$\bar{\gamma}^2$ is an
	approximate eigenvalue of $\PX$.  In other words, there exists a
	sequence of functions $\{\hgn\}_{n=1}^{\infty}$ in $L_0^2(\Pi_X)$ such
	that $\|\hgn\|_{\Pi_X} = 1$ for each~$n$, and
	\begin{equation} 
	\label{ine:PX-approx-eigen} 
	\lim\limits_{n \to
		\infty}( \PX \hgn - \bar{\gamma}^2 \hgn ) = 0 \,.
	\end{equation}
	Let
	\[
	a = \frac{2r - 1 + \sqrt{1 -  4r(1-r)(1-\bar{\gamma}^2)}}{2(1-r)\bar{\gamma}^2} \,.
	\]
	Consider the sequence of functions $\{\hgn \oplus a Q_2 \hgn \}_n$ in
	$H \subset L_0^2(\Pi)$. It is easy to show that, for each~$n$,
	\begin{equation} 
	\label{eq:PR-g+ah}
	\begin{aligned}
	\PR ( \hgn \oplus a Q_2\hgn ) &= (1-r)[\hgn + a \PX \hgn] \oplus r(a+1) Q_2\hgn \\
	&= (1-r)(1+a\bar{\gamma}^2) \hgn \oplus r(a+1) Q_2 \hgn + (1-r)a(\PX
	\hgn - \bar{\gamma}^2 \hgn) \oplus 0 \,.
	\end{aligned}
	\end{equation}
	It is straightforward to verify that
	\[
	(1-r)(1+a\bar{\gamma}^2) = \frac{r(a+1)}{a} = \frac{1 + \sqrt{1 - 4r(1-r)(1-\bar{\gamma}^2)}}{2} \,.
	\]
	Hence,~\eqref{eq:PR-g+ah} can be written as
	\[
	\PR ( \hgn \oplus a Q_2 \hgn ) - \frac{1 + \sqrt{1 - 4r(1-r)(1-\bar{\gamma}^2)}}{2} ( \hgn \oplus a Q_2 \hgn )  = (1-r)a(\PX \hgn - \bar{\gamma}^2 \hgn) \oplus 0 \,.
	\]
	By~\eqref{ine:PX-approx-eigen}, the right-hand-side goes to $0 \in H$
	as $n \to \infty$. Moreover, by~\eqref{ine:H-norm},
	$\|\hgn \oplus a Q_2\hgn\|_{\Pi}^2 \geq 1 - \bar{\gamma} > 0$.  It
	follows that
	\[
	\|\PR\|_{\Pi}  \geq \limsup_{n \to \infty}  \left\| \PR \left(
	\frac{\hgn \oplus a Q_2\hgn}{\|\hgn \oplus a Q_2\hgn\|_{\Pi}}
	\right) \right\|_{\Pi} \\ 
	= \frac{1 + \sqrt{1 - 4r(1-r)(1-\bar{\gamma}^2)}}{2} \,,
	\]
	and~\eqref{ine:rhoR-lower} holds.
\end{proof}

\section{Qualitative Relationships among Convergence Rates} 
\label{sec:qualitative}

It follows from Theorem~\ref{thm:main} that the RG sampler is $L^2$
geometrically ergodic if and only if the associated DG sampler is too;
see \pcite{roberts1997geometric} Proposition 3.2 for what is
essentially a proof of the ``if'' part. Our objective for
Section~\ref{sec:tcs} is to establish similar relations between other
pairs of component-wise samplers introduced in
Section~\ref{sec:intro}, and eventually build
Figure~\ref{fig:ergodicity}. In Section~\ref{sec:cmh} we consider a
more general setting where the Gibbs update from $\PDM$ is replaced by
a second Metropolis-Hastings update. Most of the results from
Section~\ref{sec:tcs} extend to the setting of Section~\ref{sec:cmh}.

\subsection{Two-component samplers with Gibbs updates}
\label{sec:tcs}

First, by Proposition~\ref{pro:selection}, if the RG or RC sampler is
$L^2$ geometric ergodic for some selection probability, then it is
$L^2$ geometrically ergodic for all selection probabilities.  This
allows us to treat the selection probabilities of the RG and RC
sampler as arbitrary in what follows.

It is certainly not true that, in general, $L^2$ geometric ergodicity
of the DG and RG samplers implies that of the DC or RC samplers.
The following condition on the proposal density $q$ will be useful.
\begin{condition}
\label{cond:C}
 \[
  C = \sup_{(x',x,y) \in \X \times \X \times \Y} \frac{\pi_{X|Y}(x'|y)}{q(x'|x,y)} < \infty.
  \]
\end{condition}
Condition~\ref{cond:C} is analogous to a commonly-used condition for
uniform ergodicity for full dimensional Metropolis-Hastings samplers
\citep{liu1996metropolized, mengersen1996rates, roberts1996geometric,
  smith1996exact}.  Indeed, if $\Q$ is the Metropolis-Hastings Mtk
which is reversible with respect to $\Pi_{X|Y}$ with proposal density
$q$, then under condition~\ref{cond:C}, for each $y \in \Y$, $x \in \X$, and
$A \in \mathcal{F}_X$,
\begin{equation} 
\label{ine:uniform}
\Q(A|x,y) \geq \int_A \min\left\{ \frac{q(x|x',y)}{\pi_{X|Y}(x|y)},
\frac{q(x'|x,y)}{\pi_{X|Y}(x'|y)} \right\} \pi_{X|Y}(x'|y) \, \df x'
\geq \frac{1}{C} \Pi_{X|Y}(A|y) .
\end{equation}
For a fixed $y \in \Y$, the Markov chain on~$\X$ defined by $\Q$ has
stationary distribution $\Pi_{X|Y}(\cdot|y)$ and \eqref{ine:uniform}
implies that this chain is uniformly ergodic.  Moreover,
condition~\ref{cond:C} implies that if $C$ can be calculated, then one
can use an accept-reject sampler, at least in principle, to sample
from $\Pi_{X|Y}$.

Condition~\ref{cond:C} allows us to draw the dashed arrows between DG
and DC and between RG and RC in Figure~\ref{fig:ergodicity-2b}.

\begin{proposition}
Suppose condition~\ref{cond:C} holds.  
\begin{enumerate}
\item If $\rho(\PD)<1$, then $\rho(\PDM) <1$.
\item If $\rho(\PR) < 1$, then $\rho(\PRM) < 1$.
\end{enumerate}
\end{proposition}

\begin{proof}
  Consider the first item.  By
  Corollary~\ref{cor:marginal}, $\rho(\PX)=\rho(\PD)<1$ and hence by
  Lemma~\ref{lem:marginal} $\|\PX\|_{\Pi} <1$.  Theorem~5 in
  \cite{jones2014convergence} establishes that, under
  condition~\ref{cond:C},
  $\PXM(x,A) \ge C^{-1} \PX(x,A)$ for $x \in \X$ and $A \in \mathcal{F}_X$. 
  Hence by Lemma~\ref{lem:comparison} we
  have that $\|\PXM\|_{\Pi} < 1$. An appeal to
  Lemma~\ref{lem:marginal} yields $\rho(\PDM) <1$.
  
  The second item follows from Lemma~\ref{lem:comparison}
  since Theorem~6 in \cite{jones2014convergence} establishes that,
  under condition~\ref{cond:C}, 
  $\PRM((x,y),A) \ge C^{-1} \PR((x,y),A)$ for each $(x,y)$ and measurable~$A$.
\end{proof}

Next, observe that $L^2$ geometric ergodicity of the RC sampler does
not necessarily imply that of the associated DC sampler.  Indeed, a
counter example can be constructed as follows.  Let
$\X = \Y = \{1,2\}$, and suppose that $\Pi$ is a uniform distribution
on $\X \times \Y$.  Then $\Pi_{Y|X}$ and $\Pi_{X|Y}$ are uniform
distributions on~$\Y$ and~$\X$, respectively.  Let $q(\cdot|x,y)$ be
defined with respect to the counting measure, and suppose that, for
$y \in \{1,2\}$, $q(2|1,y) = q(1|2,y) = 1$.  In other words,
$q(\cdot|x,y)$ always proposes a point in~$\X$ that is different
from~$x$.  The resulting RC chain is $L^2$-geometrically ergodic, but
the associated DC chain is periodic, and it is easy to show that
$\rhoDM = 1$.

The relations that we have described so far can be summarized in
Figure~\ref{fig:ergodicity-2a}.  As in Figure~\ref{fig:ergodicity}, a
solid arrow from one sampler to another means that $L^2$
geometric ergodicity of the former implies that of the latter, while a
dashed arrow means that the relation does not hold in general, but
does under condition~\ref{cond:C}.  A dotted arrow from one sampler to
another means that $L^2$ geometric ergodicity of the former does not
imply that of the latter in general, and we have not yet addressed
whether it does under condition~\ref{cond:C}.

\begin{figure}[h]
	\centering
	\begin{subfigure}[t]{0.23\textwidth}
		\includegraphics[width=\textwidth]{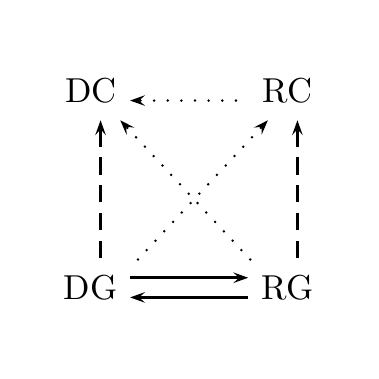}
		\caption{ } 
		\label{fig:ergodicity-2a}
	\end{subfigure}
	\begin{subfigure}[t]{0.23\textwidth}
		\includegraphics[width=\textwidth]{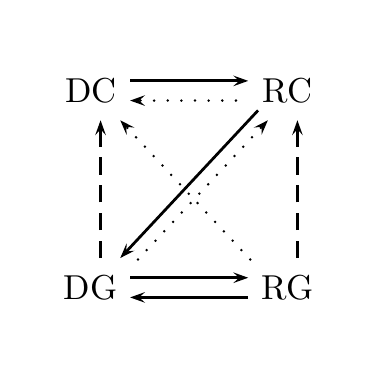}
		\caption{}
		\label{fig:ergodicity-2b}
	\end{subfigure}
	\begin{subfigure}[t]{0.33\textwidth}
		\includegraphics[width=0.97\textwidth]{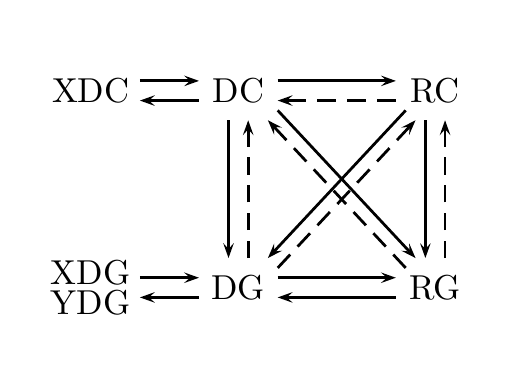}
		\caption{}
		\label{fig:ergodicity-2c}
	\end{subfigure}
	\caption{Building the relations among the convergence rates of component-wise samplers.}
\end{figure}

Lemma~\ref{lem:comparison} allows us to establish the following
result, which shows that the RC sampler is $L^2$ geometrically
ergodic whenever the DC sampler is.  This allows us to draw a solid
arrow from the DC sampler to the RC sampler in
Figure~\ref{fig:ergodicity-2b}.
\begin{proposition} 
	\label{pro:rhorm<rhodm}
	If $\rhoDM < 1$, then $\rhoRM < 1$.
\end{proposition} 
\begin{proof}
  By Lemma~\ref{lem:marginal}, $\|\PDM^2\|_{\Pi} \leq \rho(\PDM) < 1$.
  For $(x,y) \in \X \times \Y$ and
  $A \in \mathcal{F}_X \times \mathcal{F}_Y$,
	\[
	\PRM^4((x,y), A) \geq r^2(1-r)^2 \PDM^2((x,y),A) \,.
	\]
	By Lemma~\ref{lem:comparison}, $\|\PRM^4\|_{\Pi} < 1$.  Since $\PRM$
	is self-adjoint, $\rho(\PRM) = \|\PRM\|_{\Pi} = \|\PRM^4\|_{\Pi}^{1/4} < 1$.
\end{proof}

The final relation that we need to establish is given below.  It
allows us to draw a solid arrow from the RC sampler to the DG
sampler.
\begin{proposition} 
	\label{pro:rhod<rhorm}
	If $\rhoRM < 1$, then $\rhoD < 1$.
\end{proposition}
\begin{proof}
	Consider the contrapositive and recall that, by
	Lemmas~\ref{lem:marginal} and~\ref{lem:maximal},
	$\rhoD = \bar{\gamma}^2$.  Assume that $\bar{\gamma} = 1$.  It
	suffices to show that $\rhoRM = 1$.
	
	Let $g \in L_0^2(\Pi_X)$ and $h \in L_0^2(\Pi_Y)$ be such that
	$\|g\|_{\Pi_X} = \|h\|_{\Pi_Y} = 1$.  Let $f_g \in L_0^2(\Pi)$ be such
	that $f_g(x,y) = g(x)$, and, $f_h \in L_0^2(\Pi)$, $f_h(x,y) = h(y)$.
	Recall that $\rhoRM = \|\PRM\|_{\Pi}$.  By Cauchy-Schwarz,
	\[
	\begin{aligned}
	\rhoRM & \geq \langle \PRM f_h, f_g \rangle_{\Pi} \\
	&= r \int_{\X \times \Y} h(y) g(x) \Pi(\df x, \df y) + (1-r) \int_{\X
		\times \Y} \int_{\Y} h(y') \, \Pi_{Y|X}(\df y'|x) \,   g(x) \,
	\Pi(\df x, \df y) \\ 
	&= \gamma(g, h) \,.
	\end{aligned}
	\]
	Taking the supremum with respect to~$g$ and~$h$ shows that $\rhoRM \geq \bar{\gamma} = 1$.
\end{proof}

Incorporating Propositions~\ref{pro:rhorm<rhodm}
and~\ref{pro:rhod<rhorm} in Figure~\ref{fig:ergodicity-2a} yields
Figure~\ref{fig:ergodicity-2b}.  From here, one can
obtain Figure~\ref{fig:ergodicity} by following the steps below:
\begin{enumerate}
	\item When there is a path from one sampler to another consisting of only solid arrows, draw a (direct) solid arrow from the former to the latter, if there isn't one already.
	This allows us to draw solid arrows from DC to DG, from RC to RG, and and from DC to RG.
	\item When there is a dotted arrow from one sampler to another, and there is a second path from the former to the latter consisting of dashed and possibly solid arrows, convert the dotted arrow to a dashed one.
	This allows us to convert all dotted arrows in Figure~\ref{fig:ergodicity-2b} to dashed ones.
\end{enumerate}
For example, from Figure~\ref{fig:ergodicity-2b} we see that if
$\rho(\PDM) < 1$, then $\rho(\PRM) <1$ and if $\rho(\PRM) < 1$, then
$\rho(\PD) < 1$. Hence if $\rho(\PDM) < 1$, then $\rho(\PD) < 1$ and
we can obtain the solid arrow from DC to DG in
Figure~\ref{fig:ergodicity}.

Finally, we can integrate Corollary~\ref{cor:marginal} into
Figure~\ref{fig:ergodicity}, and this yields
Figure~\ref{fig:ergodicity-2c}.

\subsection{CMH with two Metropolis-Hastings updates}
\label{sec:cmh}

Consider the setting where there are two components, both of which
will be updated via Metropolis-Hastings.  Let $q_1(\cdot|x,y)$ be a
density on $\Y$ and $q_2(\cdot|x,y)$ a density on $\X$.  The deterministic-scan
conditional Metropolis-Hastings algorithm with two Metropolis-Hastings updates is now described.
	
\begin{algorithm}[H]
  \caption{Deterministic-scan CMH with two Metropolis-Hastings updates} \label{alg:tDC}
  \begin{algorithmic}[1]
    \State {\it Input:} Current value $(X_n,Y_n) = (x,y)$ \State Draw
    a random element~$W$ from $q_1(\cdot|x,y)$, and call the observed
    value~$w$.  With probability
    \[
      a_1(w;x,y) = \min \left\{ 1, \frac{\pi_{Y|X}(w|x)
          q_1(y|x,w)}{\pi_{Y|X}(y|x) q_1(w|x,y)} \right\} ,
    \]
    set $Y_{n+1} = w$; with probability $1 - a_1(w;x,y)$, set
    $Y_{n+1} = y$.  Denote the observed value of $Y_{n+1}$ by~$y'$.
    \State Draw a random element~$Z$ from $q_2(\cdot|x,y')$, and call
    the observed value~$z$.  With probability
    \[
      a_2(z;x,y') = \min \left\{ 1, \frac{\pi_{X|Y}(z|y')
          q_2(x|z,y')}{\pi_{X|Y}(x|y') q_2(z|x,y')} \right\} ,
    \]
    set $X_{n+1} = z$; with probability $1 - a_2(z;x,y')$, set
    $X_{n+1} = x$.  
   \State Set $n=n+1$.
  \end{algorithmic}
\end{algorithm}

Denote the Mtk of the two Metropolis-Hastings steps by
$P_1(\df y'|x,y)$ and $P_2(\df x'|x,y')$ so that the transition
kernel, $\PDMM$, of the CMH algorithm is formed by composing the two
Metropolis-Hastings updates:
\[
  \PDMM( (x,y), (\df x', \df y')) =  P_2(\df x'|x,y') P_1(\df y'|x,y).
\]
We will refer to the algorithm as the $\widetilde{\mbox{DC}}$ sampler.
Of course, there is a random-scan version with transition kernel
denoted $\PRMM$ which is formed by mixing the two Metropolis-Hastings
updates:
\[
\PRMM((x,y), ( \df x', \df y')) =  r P_2(\df x'|x,y) \delta_y (\df y') + (1-r) P_1(\df y'|x,y) \delta_x(\df x').
\]
We will refer to the algorithm as the $\widetilde{\mbox{RC}}$ sampler.
Notice that $\PD$ and $\PDM$ are special cases of $\PDMM$ where the
proposal density is chosen to be the appropriate full
conditional. Similarly, $\PR$ and $\PRM$ are both special cases of
$\PRMM$. Thus there may be connections between the qualitative
convergence rates of these Markov chains.  By
Proposition~\ref{pro:selection}, we may treat all selection
probabilities as arbitrary.  We begin with an extension of
Proposition~\ref{pro:rhod<rhorm}.

\begin{proposition} 
	\label{pro:rhod<rhormm}
	If $\rho (\PRMM) < 1$, then $\rhoD < 1$.
\end{proposition}

\begin{proof}
  The proof is essentially the same as that of
  Proposition~\ref{pro:rhod<rhorm} with $\PRMM$ playing the role of
  $\PRM$.  The only difference is that one needs to make use of the
  fact that $P_1(\df y'|x,y)$ leaves $\Pi_{Y|X}(\df y|x)$ invariant
  for each $x \in \X$: By Cauchy-Schwarz,
\[
\begin{aligned}
\rho (\PRMM) & \geq \langle \PRMM f_h, f_g \rangle_{\Pi} \\
&= r \int_{\X \times \Y} h(y) g(x) \Pi(\df x, \df y) + (1-r)
\int_{\X \times \Y} \int_{\Y} h(y') \, P_{1}(\df y'|x, y) \,
g(x) \, \Pi_{Y|X}(\df y|x) \Pi_X(\df x) \\
& = r \int_{\X \times \Y} h(y) g(x) \Pi(\df x, \df y) + (1-r)
\int_{\X \times \Y} h(y') \, \Pi_{Y|X}(\df y'|x) \,
g(x) \, \Pi_X(\df x) \\
&= \gamma(g, h) .
\end{aligned}
\]
\end{proof}

We will have need of the following condition on the proposal
density $q_1$ at several points.  Notice the analogy to
condition~\ref{cond:C}. 
\begin{condition}
\label{cond:C1}
 \[
    C_1 = \sup_{(x,y,y') \in \X \times \Y \times \Y}
    \frac{\pi_{Y|X}(y'|x)}{q_1(y'|x,y)} < \infty .
  \]
\end{condition}

\begin{proposition} \label{pro:rhodmm<rhodm}
	Suppose that condition~\ref{cond:C1} holds.
	If $\rho(\PDM) < 1$, then $\rho(\PDMM) < 1$.
\end{proposition}
\begin{proof}
	Observe that by condition~\ref{cond:C1} we have
	\[
	q_1(y'|x,y) a_1(y';x,y) = \pi_{Y|X}(y'|x) \min \left\{
	\frac{q_1(y'|x,y)}{\pi_{Y|X}(y'|x)},
	\frac{q_1(y|x,y')}{\pi_{Y|X}(y|x)} \right\} \ge \frac{1}{C_1} \pi_{Y|X}(y'|x),
	\]
	so that
	\[
	P_1(dy' | x, y) \ge \frac{1}{C_1} \Pi_{Y|X}(dy'|x),
	\]
	and hence,
	\[
	\PDMM((x,y), (\df x', \df y')) \ge \frac{1}{C_1} \PDM((x,y), (\df x', \df y')).
	\]
	It follows that, for $(x,y) \in \X \times \Y$ and $A \in \mathcal{F}_X \times \mathcal{F}_Y$,
	\[
	\PDMM^{2}((x,y), A) \ge \frac{1}{C^{2}_1} \PDM^{2}((x,y), A).
	\]
	By Lemma~\ref{lem:marginal}, $\|\PDM^2\|_{\Pi} \leq \rho(\PDM) <
	1$.
	Then Lemma~\ref{lem:comparison} implies that $\|\PDMM^2\|_{\Pi} < 1$.
	For $\mu \in L_*^2(\Pi)$
	and positive integer~$n$,
	\[
	\|\mu \PDMM^n - \Pi\|_{\Pi} = \|(\mu-\Pi) \PDMM^n\|_{\Pi} = \sup_{f \in L_0^2(\Pi), \|f\|_{\Pi}
		= 1} \left\langle \frac{\df \mu}{\df \Pi} - 1, \PDMM^n f \right\rangle_{\Pi}.
	\]
	Using Cauchy-Schwarz and treating even and odd~$n$ separately, we see that
	\[
		\sup_{f \in L_0^2(\Pi), \|f\|_{\Pi}
			= 1} \left\langle \frac{\df \mu}{\df \Pi} - 1, \PDMM^n f \right\rangle_{\Pi} 
		\leq \|\mu-\Pi\|_{\Pi} \|\PDMM^2\|_{\Pi}^{\lfloor n/2 \rfloor} \leq  \|\mu-\Pi\|_{\Pi} \|\PDMM^2\|_{\Pi}^{(n-1)/2},
	\]
	where $\lfloor n/2 \rfloor$ is the largest integer that does not exceed $n/2$.
	Therefore, $\rho(\PDMM) \leq \|\PDMM^2\|_{\Pi}^{1/2} < 1$.
\end{proof}

\begin{proposition}
\label{pro:rhormm<rhorm}
Suppose that condition~\ref{cond:C1} holds. If $\rhoRM < 1$, then
$\rho(\PRMM) < 1$.
\end{proposition}

\begin{proof}
  Since both $\PRM$ and $\PRMM$ are reversible,
  $\rhoRM=\|\PRM\|_{\Pi}$ and $\rho(\PRMM) = \|\PRMM \|_{\Pi}$. 
  Recall that under condition~\ref{cond:C1},
  \[
  P_1(dy' | x, y) \ge \frac{1}{C_1} \Pi_{Y|X}(dy'|x),
  \]
and hence,
\begin{align*}
\PRMM((x,y), (\df x', \df y')) & \ge r P_2(\df x'|x,y) \delta_{y}(\df y') + \frac{1-r}{C_1} \Pi_{Y|X}(dy'|x) \delta_x(\df x')  \\
 & \ge \frac{1}{C_1}
 \PRM ((x,y), (\df x', \df y')).
\end{align*}
Note that, without loss of generality, we have assumed that $\PRM$ and $\PRMM$ have the same selection probability.
The desired result now follows from Lemma~\ref{lem:comparison}.
\end{proof}

One can combine the results above with those in the previous
subsection to obtain other relations.  For example, combining
Proposition~\ref{pro:rhormm<rhorm} with
Proposition~\ref{pro:rhorm<rhodm} gives the following result.
\begin{corollary}
	Suppose that condition~\ref{cond:C1} holds.
	If $\rho(\PDM) < 1$, then $\rho(\PRMM) < 1$.
\end{corollary}

To make progress on developing further qualitative convergence
relationships, we will need to include condition~\ref{cond:C} so that
we can appeal to the results of the previous section.  Notice that the
proposal density $q_2$ from Algorithm~\ref{alg:tDC} corresponds to the
proposal density $q$ from Algorithm~\ref{alg:DC} so that
condition~\ref{cond:C} can be interpreted as a condition on $q_2$.

\begin{proposition} \label{pro:G-and-MM}
	Suppose that conditions~\ref{cond:C} and~\ref{cond:C1} hold. Then
\begin{enumerate}
\item if $\rho(\PD) < 1$, then $\rho(\PDMM) < 1$ and $\rho(\PRMM) <
  1$; 
\item if $\rho(\PR) < 1$, then $\rho(\PRMM)<1$ and $\rho(\PDMM) < 1$;
  and
\item if $\rho(\PRMM) < 1$, then $\rho(\PDMM) < 1$.
\end{enumerate}
\end{proposition}

\begin{proof}
  We consider only the first item as the others are similar.  From
  Figure~\ref{fig:ergodicity-2c} we have that, under
  condition~\ref{cond:C}, if $\rho(\PD) < 1$, then $\rho(\PDM) < 1$
  and $\rho(\PRM) < 1$.  Combining this with
  Propositions~\ref{pro:rhodmm<rhodm} and~\ref{pro:rhormm<rhorm}
  yields the claim.
\end{proof}

\begin{remark}
  The relations in Proposition~\ref{pro:G-and-MM} do not necessarily
  hold without conditions such as~\ref{cond:C} and~\ref{cond:C1}.  For
  instance, in the previous subsection we have shown that $\rho(\PRM) < 1$ does not imply $\rho(\PDM) < 1$ in general.  Since
  $\PRM$ and $\PDM$ are respectively special cases of $\PRMM$
  and $\PDMM$, 
  $\rho(\PRMM) < 1$ does not imply
  $\rho(\PDMM)<1$ in general.
\end{remark}

\begin{figure}
	\centering
	\includegraphics[width=0.35\textwidth]{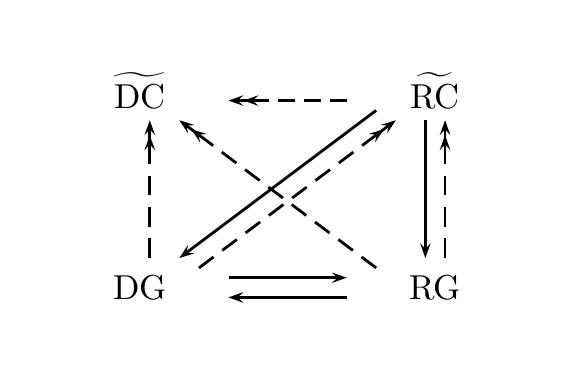}
	\caption{Known qualitative convergence relationships among the
          Gibbs samplers and their CMH variants with two
          Metropolis-Hastings updates.  } \label{fig:ergodicity-3}
\end{figure}

We depict the known qualitative convergence relations among DG,
RG, $\widetilde{\mbox{DC}}$, and $\widetilde{\mbox{RC}}$ in Figure~\ref{fig:ergodicity-3}.  A
dashed arrow with double arrowheads from one to another means that
$L^2$ geometric ergodicity of the former implies that of the latter
under conditions~\ref{cond:C} and~\ref{cond:C1}, but not in general.
Figure~\ref{fig:ergodicity-3} illustrates that the complexity of
$\widetilde{\mbox{DC}}$ (in particular, the lack of a reversible marginal Markov
chain) means that it is an open question as to whether the $L^2$
geometric ergodicity of $\widetilde{\mbox{DC}}$ implies the $L^2$ geometric ergodicity
of any of the rest.

Finally note that combining Figures~\ref{fig:ergodicity-2c}
and~\ref{fig:ergodicity-3} does not characterize all known qualitative
convergence relationships.  For example, combining the results in the
figures would suggest that under conditions~\ref{cond:C}
and~\ref{cond:C1} if $\rho(\PDM)<1$, then $\rho(\PDMM)<1$, but
Proposition~\ref{pro:rhodmm<rhodm} shows that only
condition~\ref{cond:C1} is required.

\section{Final Remarks}
\label{sec:final}

We have focused on convergence relationships between
deterministic-scan and random-scan MCMC algorithms when there are two
component-wise updates.  At the heart of these relationships is the
explicit quantitative relationship developed between Gibbs samplers in
Theorem~\ref{thm:main}.  This result is intuitively appealing since a
random-scan Gibbs sampler may update the same component consecutively
and thus one might expect convergence to be slower than the
deterministic-scan version.  When there are more than two components
it is intuitively less obvious that the random-scan Gibbs sampler will
converge substantially more slowly than the deterministic-scan version
since it is much less likely to update the same component
consecutively.  Indeed, \cite{roberts2016suprising} provide some
examples where the relationship between the convergence rates becomes
more complicated when the number of components is large. However, as
we explained in Remark~\ref{rem:more} there are technical hurdles to
investigating this rigorously.  As we saw in Section~\ref{sec:cmh}, the
situation is even more complicated, even in the two-component setting,
when considering deterministic-scan and random-scan versions of CMH
Markov chains.  There is ample room for future work along these lines.

\bigskip
{\noindent\bf Acknowledgment.}
We thank the Editor, the Associate Editor, and two anonymous reviewers for their
constructive feedback. 

\appendix

\section{Proof of Lemma~\ref{lem:marginal}} \label{app:marginal}

We will prove
\[
\|\PD^n\|_{\Pi}^{1/(n-1/2)} = \rho(\PD) = \|\PX\|_{\Pi_X} = \|\PY\|_{\Pi_Y} \,.
\]
The proof for the other equation is similar. 
\begin{enumerate}
	\item[(i)] $\|\PX\|_{\Pi_X} = \|\PY\|_{\Pi_Y}$.  This is given in
	\pcite{liu:etal:1994} Theorem 3.2.
	
	\item[(ii)] $\|\PD^n\|_{\Pi}^{1/(n-1/2)} = \|\PX\|_{\Pi_X}$.  Firstly,
	since $\PX$ is self-adjoint, for each positive integer~$n$,
	$\|\PX^n\|_{\Pi_X} = \|\PX\|_{\Pi_X}^n$  and similarly for $\PY$.

	We begin by showing that
	$\|\PD^n\|_{\Pi}^{1/(n-1/2)} \leq \|\PX\|_{\Pi_X}$.  Let
	$f \in L_0^2(\Pi)$ be such that $\|f\|_{\Pi} = 1$, and let
	\[
	\begin{aligned}
	&h_f(y) = \int_{\X} f(x,y) \Pi_{X|Y}(\df x|y)\,, \; y \in \Y \,, \\
	&g_f(x) = \int_{\Y} h_f(y) \Pi_{Y|X}(\df y|x)\,, \; x \in \X \,.
	\end{aligned}
	\]
	Then $h_f \in L_0^2(\Pi_Y)$, and $g_f \in L_0^2(\Pi_X)$.  
	Note that
	\[
	\begin{aligned}
		\langle g_f, g_f \rangle_{\Pi_X} &= \int_{\X} \int_{\Y} h_f(y) \Pi_{Y|X}(\df y|x) \int_{\Y} h_f(y') \Pi_{Y|X}(\df y'|x) \, \Pi_X(\df x) \\
		&= \int_{\Y \times \X \times \Y} h_f(y) \, h_f(y')  \Pi_{Y|X}(\df y'|x) \Pi_{X|Y}(\df x|y) \, \Pi_Y(\df y) \\
		&= \langle h_f,
		\PY  h_f \rangle_{\Pi_Y} .
	\end{aligned}
	\]
	Moreover, by the Cauchy-Schwarz inequality, $\|h_f\|_{\Pi_Y}\leq 1$.
	It follows that
	\[
	\|g_f\|_{\Pi_X}^2 = \langle h_f,
	\PY  h_f \rangle_{\Pi_Y} \leq \|\PY\|_{\Pi_Y} = \|\PX\|_{\Pi_X} \,.
	\]
	It is easy to
	verify that, for each positive integer~$n$ and
	$(x,y) \in \X \times \Y$, $\PD^n f(x,y) = \PX^{n-1} g_f(x)$.
	Therefore,
	\[
	\|\PD^n f\|_{\Pi} = \| \PX^{n-1} g_f \|_{\Pi_X} \leq
	\|\PX\|_{\Pi_X}^{n-1} \|g_f\|_{\Pi_X} \leq \|\PX\|_{\Pi_X}^{n-1/2}
	\,.
	\]
	Taking the supremum with respect to~$f$ yields the desired inequality.
	
	
	We now show that $\|\PY\|_{\Pi_Y} \leq \|\PD^n\|_{\Pi}^{1/(n-1/2)}$
	and it will follow immediately that
	$\|\PX\|_{\Pi_X} \leq \|\PD^n\|_{\Pi}^{1/(n-1/2)}$.  Let
	$h \in L_0^2(\Pi_Y)$ be such that $\|h\|_{\Pi_Y} = 1$.  Let
	$f_h \in L_0^2(\Pi)$ be such that $f_h(x,y) = h(y)$ for
	$(x,y) \in \X \times \Y$.  Then $\|f_h\|_{\Pi} = 1$.  Lastly, let
	$Q_1h \in L_0^2(\Pi_X)$ be such that
	$(Q_1h)(x) = \int_{\Y} h(y) \Pi_{Y|X}(\df y|x)$. 
	Note that, for $(x,y) \in \X \times \Y$, $(Q_1 h)(x) = (\PD f_h)(x,y)$.
	A careful calculation
	shows that
	\[
	\begin{aligned}
		\langle h, \PY^{2n-1} h \rangle_{\Pi_Y} &= \int_{\Y} h(y) \int_{\X}  (\PX^{2n-2} Q_1 h )(x) \, \Pi_{X|Y}(\df x|y) \, \Pi_Y(\df y) \\
		&= \int_{\Y \times \X} h(y) \,  (\PX^{2n-2} Q_1 h )(x) \, \Pi_{Y|X}(\df y|x) \, \Pi_X(\df x) \\
		&= \langle \PX^{n-1} Q_1 h,
		\PX^{n-1} Q_1h \rangle_{\Pi_X} \\
		&= \langle \PD^n f_h, \PD^n f_h
		\rangle_{\Pi} \\
		& \leq \|\PD^n\|_{\Pi}^2 \,.
	\end{aligned}
	\] 
	Since $\PY^{2n-1}$ is non-negative definite,
	\[
	\|\PY\|_{\Pi_Y}^{2n-1} = \|\PY^{2n-1}\|_{\Pi_Y} = \sup \{ \langle \PY^{2n-1} h', h'
	\rangle_{\Pi_Y}: h' \in L_0^2(\Pi_Y), \|h'\|_{\Pi_Y} = 1\} \,.
	\]
	This shows that $\|\PY\|_{\Pi_Y}^{2n-1} \leq \|\PD^n\|_{\Pi}^2$.
	
	\item[(iii)] $\rho(\PD) = \|\PX\|_{\Pi_X}$.  By
	Lemma~\ref{lem:roberts}, $\|\PX\|_{\Pi_X} = \rho(\PX)$, the $L^2$
	convergence rate of the $X$-marginal DG chain.
	
	We now show that $\rho(\PX) \leq \rhoD$.  Let $g \in L_0^2(\Pi_X)$
	be such that $\|g\|_{\Pi_X} = 1$.  Let $f_g \in L_0^2(\Pi)$ be such
	that $f_g(x,y) = g(x)$.  Then $\|f_g\|_{\Pi} = 1$.  For any
	$\mu \in L_*^2(\Pi_X)$ and positive integer~$n$,
	\[
	|\mu \PX^n g - \Pi_X g | = |\tilde{\mu} \PD^n f_g - \Pi f_g |
	\leq \|\tilde{\mu} \PD^n - \Pi \|_{\Pi} \,,
	\]
	where $\tilde{\mu}$ is any measure in $L_*^2(\Pi)$ such that
	$\int_{\Y} \tilde{\mu}(\cdot, \df y) = \mu(\cdot)$.  Taking the
	supremum with respect to~$g$ shows that
	\[
	\|\mu \PX^n - \Pi_X\|_{\Pi_X} \leq \|\tilde{\mu} \PD^n - \Pi
	\|_{\Pi} \,.
	\]
	This implies that $\rho(\PX) \leq \rhoD$.
	
	Finally, we show that $\rhoD \leq \rho(\PX)$.  Let
	$\tilde{\mu} \in L_*^2(\Pi)$, and define $f \in L_0^2(\Pi)$ and
	$g_f \in L_0^2(\Pi_X)$ as in (ii).  Then, for a positive
	integer~$n$,
	\[
	|\tilde{\mu} \PD^n f - \Pi f | = |\mu \PX^{n-1} g_f - \Pi_X g_f|
	\leq \|\mu \PX^{n-1} - \Pi_X\|_{\Pi_X} \,,
	\]
	where $\mu(\cdot) = \int_{\Y} \tilde{\mu}(\cdot, \df y)$.  Taking
	the supremum with respect to~$f$ shows that
	\[
	\|\tilde{\mu} \PD^n - \Pi \|_{\Pi} \leq \|\mu \PX^{n-1} -
	\Pi_X\|_{\Pi_X} \,,
	\]
	which implies that $\rhoD \leq \rho(\PX)$.
\end{enumerate}

\section{Proof of Proposition~\ref{pro:l1l2}} 
\label{app:l1l2}

We will prove the result for $\PDM$ and $\PRM$. The proofs for $\PD$ and
$\PR$ are similar.
We will make use of results in \cite{roberts2001geometric}.
These results require $\mathcal{F}_X$ and $\mathcal{F}_X \times \mathcal{F}_Y$ to be countably generated.
We have assumed that $\mathcal{F}_X$ and $\mathcal{F}_Y$ are countably generated.
This implies that $\mathcal{F}_X \times \mathcal{F}_Y$ is also countably generated.
Indeed, if $\mathcal{F}_X$ is generated by $\{A_i\}_{i=1}^{\infty}$ and $\mathcal{F}_Y$ is generated by $\{B_j\}_{j=1}^{\infty}$, then $\mathcal{F}_X \times \mathcal{F}_Y$ can be generated by sets of the forms $A_i \times \Y$ and $\X \times B_j$.

Consider $\PRM$. Then the claim follows immediately due to its
reversibility with respect to $\Pi$ \cite[][Theorem
2]{roberts2001geometric}.

Now suppose $\PDM$ is $L^2$ geometrically ergodic.  Then it is
$\Pi$-a.e. geometrically ergodic \cite[][Theorem
1]{roberts2001geometric}. Conversely, suppose that the $\PDM$ is
$\Pi$-a.e. geometrically ergodic.  This implies $\PXM$ is
$\Pi_X$-a.e. geometrically ergodic.  It is also straightforward to
check that $\PXM$ is $\varphi^*$-irreducible, with
$\varphi^*(\cdot) = \int_{\Y} \varphi(\cdot, \df y)$.  Since $\PXM$
is reversible with respect to $\Pi_X$, it is also $L^2$
geometrically ergodic.  By Corollary~\ref{cor:marginal}, $\PDM$ must be $L^2$ geometrically ergodic
as well.

\section{Proof of Lemma~\ref{lem:H-closed}} 
\label{app:H-closed}

It suffices to show that~$H$ is closed \citep[see, e.g.,][\S
6]{helmberg2014introduction}.  Consider a sequence of functions
in~$H$, $\{g_n \oplus h_n\}_{n=1}^{\infty}$, such that
\[
\lim\limits_{n \to \infty} (g_n \oplus h_n) = f \in L_0^2(\Pi) \,.
\]
The sequence $\{g_n \oplus h_n\}$ is Cauchy, that is,
\[
\lim\limits_{n \to \infty} \sup_{m \geq n} \|g_n \oplus h_n - (g_m
\oplus h_m) \|_{\Pi} = 0 \,.
\]
By~\eqref{ine:H-norm},
\[
\|g_n \oplus h_n - (g_m \oplus h_m) \|_{\Pi}^2 \geq
(1-\bar{\gamma})(\|g_n - g_m\|_{\Pi_X}^2 + \|h_n - h_m\|_{\Pi_Y}^2)
\,.
\]
Since $\bar{\gamma} < 1$, $\{g_n\}$ and $\{h_n\}$ are Cauchy as well.
By the completeness of $L_0^2(\Pi_X)$ and $L_0^2(\Pi_Y)$, there exist
$g \in L_0^2(\Pi_X)$ and $h \in L_0^2(\Pi_Y)$ such that
\[
\lim\limits_{n \to \infty} g_n = g \,, \quad \lim\limits_{n \to
	\infty} h_n = h \,.
\]
Again by~\eqref{ine:H-norm},
\[
\|g_n \oplus h_n - (g \oplus h )\|_{\Pi}^2 \leq
(1+\bar{\gamma})(\|g_n-g\|_{\Pi_X}^2 + \|h_n-h\|_{\Pi_Y}^2) \,.
\]
This implies that
\[
\lim\limits_{n \to \infty} (g_n \oplus h_n) = g \oplus h \,.
\]
Hence, $f = g \oplus h \in H$, meaning that~$H$ is closed.
%
%

\section{A Lemma concerning Theorem~\ref{thm:main}} 
\label{app:main}

The following lemma is the result of several elementary facts in
functional analysis.  See, e.g., \cite{helmberg2014introduction}, \S
23, 24.
\begin{lemma} \label{lem:approxeigen} Let $H'$ be a real or complex
	Hilbert space equipped with inner product
	$\langle \cdot, \cdot \rangle$ and norm $\|\cdot\|$.  Let~$P$ be a
	bounded non-negative definite operator on $H'$.  Then $\|P\|$ is an
	approximate eigenvalue of~$P$, i.e., there exists a sequence
	$\{f_n\}_{n=1}^{\infty}$ in~$H'$ such that $\|f_n\| = 1$ for
	each~$n$, and $\lim\limits_{n \to \infty} \|Pf_n - \|P\|f_n\| = 0$.
\end{lemma}
\begin{proof}
	Since~$P$ is non-negative definite,
	\[
	\|P\| = \sup_{f \in H', \|f\| = 1} \langle Pf, f \rangle \,.
	\]
	It follows that there exists a sequence $\{f_n\}_n$ in~$H'$ such
	that $\|f_n\| = 1$ for each~$n$, and
	$\lim\limits_{n \to \infty} \langle Pf_n, f_n \rangle = \|P\|$.
	Note that
	\[
	\langle Pf_n, f_n \rangle \leq \|Pf_n\| \leq \|P\| \,.
	\]
	This implies that $\|Pf_n\| \to \|P\|$ as $n \to \infty$.  It
	follows that
	\[
	\lim\limits_{n \to \infty} \|Pf_n - \|P\|f_n\|^2 = \lim\limits_{n
		\to \infty} ( \|Pf_n\|^2 + \|P\|^2 - 2 \|P\| \langle Pf_n, f_n
	\rangle ) = 0 \,.
	\]
\end{proof}

\bibliographystyle{apalike}
\bibliography{qqgj}

\end{document}